\documentclass[11pt,reqno]{amsart}

\usepackage{amsmath}
\usepackage{amsthm}
\usepackage{graphicx}
\usepackage{upgreek}
\usepackage{hyperref}
\usepackage{mathrsfs}
\usepackage{xcolor}
\usepackage{bm,bbm}
\usepackage{amsfonts}
\usepackage{amssymb}
\usepackage{csquotes}
\usepackage{stmaryrd}
\usepackage{ulem}
\usepackage{lipsum}

\newcommand{\sredm}[1]{\ifmmode\text{\R^dout{\ensuremath{\displaystyle \textcolor{red}{#1}}}}\else\sout{\textcolor{red}{#1}}\fi}

\usepackage{enumerate}

\usepackage{soul}
\usepackage[margin=1in]{geometry}
\numberwithin{equation}{section}

\newtheorem{theorem}{Theorem}[section]
\newtheorem{lemma}{Lemma}[section]
\newtheorem{assumption}{Assumption}[section]
\newtheorem{proposition}{Proposition}[section]

\newtheorem{example}{Example}[section]
\newtheorem{definition}{Definition}[section]
\newtheorem{remark}{Remark}[section]

\DeclareMathOperator*{\argmax}{arg\,max}

\newcommand{\G}{{\mathbb G}}
\newcommand{\X}{{\mathbb X}}

\newcommand{\EE}{{\mathbb E}}
\newcommand{\FF}{{\mathbb F}}
\newcommand{\PP}{{\mathbb P}}\newcommand{\E}{\mathbb{E}}
\newcommand{\R}{\mathbb{R}}

\newcommand{\N}{\mathbb{N}}

\newcommand{\calA}{{\mathcal A}}
\newcommand{\calB}{{\mathcal B}}
\newcommand{\calC}{{\mathcal C}}

\newcommand{\calF}{{\mathcal F}}
\newcommand{\calG}{{\mathcal G}}

\newcommand{\calI}{{\mathcal I}}

\newcommand{\calK}{{\mathcal K}}
\newcommand{\calL}{{\mathcal L}}
\newcommand{\calM}{{\mathcal M}}

\newcommand{\calP}{{\mathcal P}}

\newcommand{\calY}{{\mathcal Y}}

\newcommand{\la}{\lambda}
\newcommand{\sig}{\sigma}

\newcommand{\eps}{\varepsilon}

\newcommand{\al}{\alpha}
\newcommand{\gam}{\gamma}

\newcommand{\skp}{\vspace{\baselineskip}}

\newcommand\iy{\infty}

\newcommand{\one}{\mathbbm{1}}

\newcommand{\Fis}{F_{i,s}}
\newcommand{\Fs}{F_{s}}
\newcommand{\Ft}{F_{t}}
\newcommand{\Gis}{G_{i,s}}
\newcommand{\Gs}{G_{s}}
\newcommand{\Gt}{G_{t}}

\newcommand{\noi}{\noindent}

\newcommand{\lbrs}[1]{\llbracket #1 \rrbracket}

\newcommand{\red}{\textcolor{black}}

\newcommand{\Vfour}{\bar{V}}

\title[Existence of Optimal Stationary Singular Controls and MFG Equilibria]{Existence of Optimal Stationary Singular Controls and Mean Field Game Equilibria}\thanks{* This is the final version of the paper. To appear in {\it Mathematics of Operations Reearch.}}

\author[A. Cohen]{Asaf Cohen }
\address{Department of Mathematics\\
University of Michigan\\
Ann Arbor, MI 48109\\
United States
}
\email{shloshim@gmail.com }

\author[C. Sun]{Chuhao Sun}
\address{Department of Mathematics\\
University of Michigan\\
Ann Arbor, MI 48109\\
United States
}
\email{chuhaos@umich.edu}

\date{\today}

\begin{document}

\maketitle

\begin{abstract}
In this paper, we examine the stationary relaxed singular control problem within a multi-dimensional framework for a single agent, as well as its mean field game equivalent. We demonstrate that optimal relaxed controls exist for \red{two problem classes: one driven by queueing control and the other by harvesting models}. These relaxed controls are defined by random measures across the state and control spaces, with the state process described as a solution to the associated martingale problem. By leveraging findings from [Kurtz and Stockbridge, {\it Electron.~J.~Probab.}, 2001]
, we establish the equivalence between the martingale problem and the stationary forward equation. This allows us to reformulate the relaxed control problem into a linear programming problem within the measure space. We prove the sequential compactness of these measures, thereby confirming the feasibility of achieving an optimal solution. Subsequently, our focus shifts to mean field games. Drawing on insights from the single-agent problem and employing Kakutani--Glicksberg--Fan fixed point theorem, we derive the existence of a mean field game equilibria.


\skp
\noi{\bf Keywords:} Singular control, stationary control,  mean field games, Brownian control problems.

\noi{\bf AMS subject classification:} 
%
%
93E20,  	
60G10 
91A16 
49N80 
%
%
%
%

\end{abstract}

\section{Introduction}
In this paper, we establish the existence of optimal controls for stationary singular single-agent control problems, as well as the existence of mean field game (MFG) equilibria for stationary singular MFGs.

\subsection{The single-agent control problem} 

We consider relaxed singular control problems involving stationary controls. We address two distinct problems in our study. Firstly, we consider \red{\it{queueing-inspired problem}}, which is traditionally formulated as a cost-minimizing problem where the underlying state process is not constrained, although it may possess a reflection boundary. Such a model finds relevance in queueing problems, particularly under heavy traffic conditions. Secondly, we explore a \red{\it{harvesting-inspired problem}}, which is traditionally formulated as a reward-maximization problem. Here, the underlying state process—representing the population, for instance—is confined to the non-negative orthant. In the queueing-inspired problem, we assume a Lyapunov condition for the uncontrolled process, ensuring stability. In the harvesting-inspired problem, we avoid such an assumption, yet we obtain stability by the near-monotone property of the reward function. Our work culminates in establishing the existence of optimal relaxed stationary singular controls for both problem formulations.

\subsubsection{Model outline, key results, and proof techniques} Initially, for the underlying dynamics for both models, we leverage the characterization provided by Kurtz and Stockbridge \cite{kur-sto}, which establishes a link between the relaxed singular control and two measures: the stationary distribution of the controlled process, and the occupation (joint) measure of both the controlled process and the singular control. While \cite{kur-sto} addresses the characterization for one-dimensional singular controlled diffusion, we extend this insight to the multi-dimensional domain. Afterward, we leverage this characterization to formulate an equivalent linear programming problem over a space of measures. Ultimately, we utilize the linear programming framework to establish the existence of a solution to a specific martingale problem, which is the recovery of the relaxed singular control. 
We will now elaborate on the model, present the key results, and outline our proof techniques for the single-agent problem in more detail. 

For the sake of clarity, we begin by presenting the optimal control problem in its strict formulation. Consider the $d$-dimensional controlled state process:
\begin{align}\label{eq:strict}
X_t=X_0+\int_0^t\beta(X_s)ds+\int_0^t\sigma(X_s)d W_s+ \int_0^t \phi(X_s)d\Fs-\int_0^t\gamma(X_s)d\Gs.
\end{align}
The noise in the system is driven by the $d_1$-dimensional Wiener process $W$. The multidimensional processes $F$ and $G$ are adapted to the underlying filtration, with increments in nonnegative orthants (possibly, with different dimensions), denoted by $\FF=\R_+^{d_2}$ and $\G=\R_+^{d_3}$, respectively. These increments could be singular with respect to the Lebesgue measure. The process $F$ is termed the {\it intrinsic singular component} and is constituting an inherent component of the system that remains uncontrolled. On the other hand, the process $G$ represents the {\it control} selected by the agent. Both of these processes are constrained to exhibit stationary increments with respect to a stationary distribution for $X$, denoted by $\nu^G$ in order to emphasize the dependence on the chosen control $G$. An illustrative instance of an intrinsic component is the reflection of the process on a boundary of a given domain (e.g., nonnegative orthant or a box). The objective of the agent is to optimize the expected cost or reward, computed under the initial stationary distribution:
\begin{align*}
\EE_{\nu^G}\Big[\int_0^1\ell(X_t)dt+\int_0^1f(X_t)\cdot dF_t+\int_0^1g(X_t)\cdot d\Gt\Big].
\end{align*}
Given our focus on stationary controls, we restrict our analysis to the time interval $[0,1]$ without loss of generality, \red{since starting from the stationary distribution, the system's behavior becomes independent of the initial time. This simplification allows us to consider a representative time interval without affecting the results.}

Having presented the dynamics in a straightforward manner using strict controls, we are now ready to proceed with the relaxed formulation. In this case both singular components (intrinsic and control) are defined via random measures, which on high level provide randomization for the singular components. Specifically, {\it relaxed controls} are defined using the collection of random measures $\Gamma=\{\Gamma_i\}_{i=d_2+1}^{d_2+d_3}$ over $\X\times\G\times[0,1]$, 
where: $\X\subseteq\R^d$ is the state space for the state process $X$, $\G=\R_+^{d_3}$ is a nonnegative orthant, standing for the set of possible increments for the control $G$, and $[0,1]$ is the time interval. Similarly, {\it relaxed intrinsic singular components} $\Phi=\{\Phi_i\}_{i=1}^{d_2}$ are random measures over $\X\times\FF\times[0,1]$. We  require that under $\Gamma$, the state process $X$ has a stationary distribution (to be denoted by) $\nu^\Gamma$, where the state process is defined via a martingale problem in such a way that the following process is a martingale
\begin{align*}
\calM^{\Gamma, h}_t:=h(X_t)-\int_0^t(\calA h)(X_s)ds-\sum_{i=1}^{d_2}\int_{\X\times\FF\times[0,t]}(\calB_i h)(x,z)\Phi_i(dx,dz,ds)\\-\sum_{i=d_2+1}^{d_2+d_3}\int_{\X\times\G\times[0,t]}(\calB_i h)(x,y)\Gamma_i(dx,dy,ds),
\end{align*}
for any test function $h\in\calC^2_b(\X)$, the set of twice continuously differentiable functions whose first derivative vanishes outside a compact domain. 
For $h\in\calC^2_b({\X})$ we let $\calA h:{\X}\to\R$ and for $1\le i\le d_2$, $\calB_i h:{\X}\times\FF\to\R$, while for $d_2+1\le i\le d_2+d_3$, $\calB_i h:{\X}\times\G\to\R$,  be defined as:
\begin{align}\notag
&(\calA h)(x):=\frac{1}{2}\sum_{i,j=1}^da_{i,j}(x)\frac{\partial^2 h}{\partial x_i\partial x_j}(x)+\sum_{i=1}^d \beta_i(x)\frac{\partial h}{\partial x_i}(x),
\\\notag
&(\calB_i h)(x,z):= \frac{h(x+ \phi(x)z)-h(x)}{z_i}, \qquad\text{if } \;  z_i>0, \qquad\text{for }\; 1\le i\le d_2,
\\\notag
&(\calB_i h)(x,y):= \frac{h(x- \gamma(x)y)-h(x)}{y_i}, \qquad \text{if } \; y_i>0, \qquad\text{for }\; d_2+1\le i\le d_2+d_3.
\end{align}
where 
$a(x):=\sigma(x)\sigma^T(x)$. 
We use the standard extension $(\calB_i h)(x,0):=-\gamma^T_i(x)\nabla h(x)$ or $\phi^T_i(x)\nabla h(x)$, where $\gamma^T_i(x)$ is the $i$-th row of $\gamma^T(x)$, and similarly for $\phi^T_i(x)$.

We show in Section \ref{sub32:min} (Lemmas \ref{lem:exo} and \ref{lem:exi}) that the controlled process can be characterized using the operators, 
i.e., there exist finite Borel measures on $\X\times\FF$, denoted by $\lambda^\Gamma= \{\lambda^\Gamma_i\}_{i=1}^{d_2}$, and on ${\X}\times\G$, denoted by $\mu^\Gamma= \{\mu^\Gamma_i\}_{i=d_2+1}^{d_2+d_3}$, such that for all $h\in\calC^2_b({\X})$,  
\begin{align}\label{mar_prm}
\int_{{\X}}\calA h(x)\nu^\Gamma(dx)+\sum_{i=1}^{d_2}\int_{{\X}\times\FF}(\calB_i h)(x,z)\lambda^\Gamma_i(dx,dz)+\sum_{i=d_2+1}^{d_2+d_3}\int_{{\X}\times\G}(\calB_i h)(x,y)\mu^\Gamma_i(dx,dy)=0.
\end{align}
Each of the {\it occupation measures} in the families $\{\lambda_i\}_{i=1}^{d_2}$ and $\{\mu_i\}_{i=d_2+1}^{d_2+d_3}$ are the expectation (with respect to $\nu^\Gamma$) of the corresponding random measures in $\Phi=\{\Phi_i\}_i$ and $\Gamma=\{\Gamma_i\}_i$ over the time interval $[0,1]$.  
Our aim is then to solve the {\it linear programming problem}
\begin{align*}
&\text{optimize}\qquad J(\Gamma):=\int_{\X}\ell(x)\nu^\Gamma(dx)+\int_{\X\times\FF}f(x)\cdot\la^\Gamma(dx,dz)+\int_{\X\times\G}g(x)\cdot \mu^\Gamma(dx,dy)\\
&\text{such that}\qquad (\nu^\Gamma,\mu^\Gamma,\la^\Gamma) \;\text{satisfies}\; \eqref{mar_prm}.
    \end{align*}
We argue in Remark \ref{rem:mu:la} that, without loss of generality, we may merge the singular control $\Gamma$ with the singular component $\Phi$ and as a result merge $\mu$ and $\lambda$ as well; hence ignoring $\lambda$ for the rest of this subsection.
The analysis of identifying optimal relaxed stationary singular controls and the MFG equilibria hinges on this linear programming representation.

With the linear programming representation established, we demonstrate in Section \ref{sub33:min} that the optimum is finite\footnote{\red{While every minimization problem can be reformulated as a maximization problem by flipping the signs, and vice versa, we retain the terminologies `maximization' and `minimization' to reflect the distinct contexts in which these terms are commonly used.}} \red{for the queueing-inspired (minimization) problem, and in Section \ref{sub33:max} for the harvesting-inspired (maximization) problem}, as outlined in Lemma \ref{ass:near:min} and Lemma \ref{ass:near} respectively. Subsequently, the crux of our technical endeavor unfolds in Section \ref{sec:34:min} and Section \ref{sec:34:max} \red{for the two cases respectively}, where we establish the sequential compactness of the set of measures $\{\nu^\Gamma\}_\Gamma$ and $\{\mu^\Gamma\}_\Gamma$, satisfying \eqref{mar_prm} (with $\lambda^\Gamma$ being ignored). This pivotal result, outlined in Proposition \ref{measure:compact:min} and Proposition \ref{measure:compact:max}, is attained by demonstrating the tightness of the measures. The strategies employed diverge between the two cases. In the queueing-inspired case, we employ the technique of one-point compactification. Namely, we expand the measures to compactified spaces straightforwardly and, by going to a subsequence if necessary, we pick the limit point. The subsequent task is to demonstrate that the limiting measures maintain support on the original spaces, rendering the original collection of measures sequentially compact. For this, we take advantage of the near-monotone structure of the cost functions $\ell$, $g$, and $f$, imposed in this case, adopting methodologies akin to those presented in \cite{Borkar}. 
However, in the maximization case, where rewards are assumed to be bounded, this approach proves inadequate, necessitating the utilization of uniform probabilistic boundedness properties of the process. Leveraging Lyapunov functions, we establish the uniform boundedness in expectation of the control $\Gamma$. Finally, in the two cases, the sequential compactness result implies Theorem \ref{thm:main:min} and \red{Theorem \ref{thm:main:max}}, which are the main theorems in the single-agent problem, asserting the attainability of the optimum.

\subsubsection{Literature review} We will now review some literature on 
singular control problems and optimal control problems with stationary controls. In essence, singular controls are characterized by their singularity with respect to the Lebesgue measure, occasionally leading to discontinuities in the dynamics of the process. These controls find diverse applications, such as in mathematical finance, where they represent consumption, investment, and dividend policies (e.g., \cite{davis1990portfolio, reppen2018singular}), in biology, where they model population dynamics alongside seeding and harvesting strategies (e.g., \cite{alv-hen2019, HNUK19, hen-qua2020, cohen2021optimal}), and in manufacturing and queueing networks, where singular controlled processes approximate systems under heavy traffic by representing customer rejections and idle periods (e.g., \cite{Atar-Budh-Will-07, bud-gho-chi2011, coh2019a, ata2023singular}).

In the domain of stationary control problems, several notable contributions exist. For instance, in \cite{klein69}, the author delves into stationary control problems within linear systems, considering scenarios where control-dependent noise is a factor. Building upon this foundation, \cite{Hau71} extends the investigation to models featuring both state and control-dependent noise. Furthermore, in \cite{e2018class}, optimality conditions are provided for a specific class of stationary singular control problems within a linear diffusion framework. Additionally, \cite{RL22} explores the realm of optimal stationary control through the lens of reinforcement learning methodologies.

Stationary control problems are often associated with infinite horizon control problems featuring ergodic costs/rewards, where the optimization criterion takes the form of a ``limit of long-time average". In ergodic control problems, admissible controls are typically restricted to stationary controls, and the time average often coincides with the space average according to the stationary distribution, regardless of the initial state of the system. For problems with ergodic criteria and singular control in biology see e.g., \cite{alv-hen2019, lia-zer2020, cohen2021optimal} and in queueing systems see, e.g., \cite{Pang15,Pang16,biswas2017,Pang18,Pang19} and the references therein. \red{In order to establish a connection between ergodic costs and stationary control problems, one typically requires additional structural characteristics in the problem that lead to continuity features. Namely, that the admissible controls, or the candidates for optimal controls make the controlled process satisfy \cite[Lemma 4.9]{kha2}, which says that the transition kernel is uniformly continuous with respect to the initial value. This is not necessarily true in general in our setting, as singular controls can cause jumps. However, we refrain from imposing further structure on the problem and defer such analysis to future research endeavors.}

Discounted control problems, on the other hand, are also associated with stationary controls but are distinct from ergodic control problems as they heavily depend on the initial state of the system. Despite the system eventually converging to its stationary distribution as time approaches infinity, the linear-programming structure is no longer valid due to the initial state's significant impact on the cost/reward.

As mentioned above, our formulation of the problem and its analysis, relies on the characterization of the stationary singular control by two measures, which leads to a linear programming problem. This link dates back to Echeverr\'ia \cite{echeveria} who gave a criterion for the stationary distribution of Markov process. Kurtz and Stockbridge \cite{kurtz98} extended this characterization for the controlled martingale problem and thus established the equivalence between the control problem and a linear programming problem over a space of measures \cite{kurtz2017linear} in a one-dimensional setup. Budhiraja \cite{Budh-03} established a similar characterization in cases where the state dynamics are constrained to solve an optimal ergodic regular control problem. In his study, the state constraint was employed to derive stability results. Alongside the characterization, which establishes a connection between the control and the stationary distribution of the state process, he demonstrated optimality through the weak limit of measures. Kang and Ramanan  \cite{weining} characterized the stationary distribution when the underlying dynamics are reflected diffusions. 

\subsubsection{Contextualizing our single-agent problem contribution} 
We now position our contribution in terms of modeling, results, and techniques within the broader context of existing research on single-agent singular control problems. Our setting captures two distinct classes of multi-dimensional models: one from from queueing theory (Example \ref{rem:rbm}) formulated with cost minimization, and another mathematical biology, specifically the harvesting model in population dynamics (Example \ref{lok-vol}) focused on reward maximization. These two classes of models have features that fail to satisfy the assumptions in some classical papers. Namely, one of the classes of the queueing control models that we cover does not assert a penalty for idleness. This condition translates to a $0$ cost associated with a certain singular component. Namely, the singular cost function is not necessarily strictly positive, and specifically, not bounded away from $0$, as required by some of the papers considering singular control problems. On the other hand, the harvesting model considers population dynamics, which by its nature has a quadratic drift coefficient; this drift is neither bounded nor Lipschitz, and such conditions are often seen in papers that do not specifically deal with biology models. Our setting, however, is capable of incorporating these two models into the framework.

As mentioned earlier, following the characterization result from \cite{kur-sto}, Kurtz and Stockbridge \cite{kurtz2017linear} further study a stationary singular control problem via the linear programming formulation and prove the existence of optimal singular controls. 
However, they only studied the cost minimization problem, not the reward maximization, which is not a trivial reversal. The relaxed singular controls in their paper, as of \cite{kur-sto}, are only one-dimensional, while we allow multiple dimensions and could incorporate more general applications, for example, see Remark \ref{Up-down} for an example of controls in both directions. In addition, they do not show that the optimal cost is finite while we do, and further require the singular cost or the budget constraint to be bounded away from $0$. We allow the singular cost to be equal to $0$ (possibly within a compact set), which allows us to cover the queueing models from Example \ref{rem:rbm} mentioned in the previous paragraph. 

In \cite{kruk1,kruk2}, Kruk considers both discounted and ergodic problems, when the state process is limited to a Brownian motion with reflections, and shows the smoothness of the value function, as well as the existence and regularity of the optimal reflected domain, via geometric approach. 

In a sequence of three papers Haussmann and Suo consider finite horizon singular control problems, using tools such as: the martingale problem and compactification arguments in \cite{Haussmann2}, the dynamic programming principle to characterize the value function as a viscosity solution of the HJB equation in \cite{Haussmann2}, and a time change method, to transform the singular control problem into a new problem that only includes regular controls, under the Roxin conditions \cite{Haussmann1}. 
In the latter, the singular control takes a specific form, that in each dimension of the dynamic, $i\in\lbrs{d}$, it has two opposite directions: plus, $+\nu^1_i$, and minus, $-\nu^2_i$. Our setting is more general and covers this example.  

It is worth noting that our linear programming approach has its roots in the martingale problem, which is considered in \cite{Haussmann2}; yet, our methodology and compactness analysis are different since these three papers are for finite horizon problems. Moreover, these papers require the drift and diffusion coefficients of the state dynamic to be bounded and continuous, while we do not impose such constraints; specifically, the unboundedness is crucial, as in the case of the harvesting model in population dynamics considered in Example \ref{lok-vol}, both the drift and the diffusion are unbounded, and in the Ornstein--Uhlenbeck process Example \ref{OU-pro}, the drift is unbounded. 

A different time rescaling method is used in Budhiraja and Ross \cite{bud-ros2006} to study an optimal singular control problem with state constraint, under the discounted criteria. This method is linked to the weak M1 topology \cite{cohen2021optimal}. 

In the paper by Menaldi and Robin \cite{menaldi-robin}, the authors consider a singular control problem for multidimensional Gaussian-Poisson processes, under both ergodic and discounted settings, using HJB equations. They establish the convergence of the discounted cost to the ergodic cost, and several properties of the potential function (solution to HJB). However, despite the presence of Poisson jumps, their dynamic is of affine structure, i.e., the drift and diffusion are affine functions of the state, or constant, while we allow for more general models, again, for example, in the harvesting model in population dynamics, Example \ref{lok-vol}, the drift is quadratic.

\subsection{The mean field game}
The study of MFGs dates back to the pioneering work by Lasry and Lions \cite{MFG1,MFG2}, and Huang, Caines and Malham{\'e}, e.g., \cite{Huang2006,Huang2007}. Generally speaking, MFGs approximate Nash equilibria in many-player games where the interaction between the players is ``weak", that is, through the empirical distribution of the players' states. For some theoretical studies of MFGs, see \cite{CD1,CD2}. The theory also has applications in various fields, e.g., finance \cite{car-lac2017, li-rep-sir2019}, economics \cite{gra2016, car2020}, public health, \cite{ eli-hub-tur2020}, or networks, \cite{Lei, Vijay}.

\subsubsection{Model outline, key results, and proof techniques} We now describe our method to define and solve the MFG. The proof of the existence of an MFG equilibrium is established via Kakutani--Glicksberg--Fan fixed point theorem for an MFG with a linear programming structure. Specifically, we consider the same dynamics as in the single-agent model, and for any fixed measure $\tilde \nu$, and any relaxed stationary singular control $\Gamma$, we define the linear programming problem:
\begin{align*}
&\text{optimize}\qquad J(\Gamma,\tilde\nu):=\int_{\X}\ell(x,\tilde\nu)\nu^\Gamma(dx)+\int_{\X\times\FF}f(x,\tilde\nu)\cdot\la^\Gamma(dx,dz)+\int_{\X\times\G}g(x,\tilde\nu)\cdot \mu^\Gamma(dx,dy)\\
&\text{such that}\qquad (\nu^\Gamma,\mu^\Gamma,\la^\Gamma) \;\text{satisfies}\; \eqref{mar_prm}.
    \end{align*}
Formally, we define an MFG equilibrium as a measure $\tilde\nu$ for which there exists a relaxed control $\Gamma$ that optimizes $J(\Gamma,\tilde\nu)$ and such that $\nu^\Gamma=\tilde\nu$. 
The existence of an optimal control for the single-agent problem plays a key role in the analysis here and is the main challenge in the establishment of MFG equilibria. It implies that the optimum above is achieved for both the maximization and minimization problems. Having this at hand, we use Kakutani--Glicksberg--Fan fixed point theorem as follows.

As before, we may combine the two singular components and their associated measures, hence, ignoring $\lambda$ for the rest of this subsection. We start by establishing a mapping $\Psi$ that maps the stationary distribution of the controlled process $\nu$ and the occupation measure components $\mu:=\{\mu_i\}_{i=d_2+1}^{d_2+d_3}$ from \eqref{mar_prm} to the cost/reward of the controlled process. Subsequently, we demonstrate the continuity of this mapping. Leveraging Berges's maximum theorem, we establish that the set-valued mapping $\Psi^*$, which maps $(\tilde\nu,\tilde\mu)$ to $\argmax \Psi$, is upper hemicontinuous and compact-valued. To apply Berges's maximum theorem, it is imperative that the set of measures $(\nu,\mu)$ satisfying \eqref{mar_prm} is compact, a condition we address through the sequential compactness results elucidated in Section \ref{sec:34:max}. Subsequently, the mapping $\Psi^*$ is a Kakutani map, enabling the application of Kakutani--Glicksberg--Fan fixed point theorem to ascertain the existence of a fixed point, which constitutes our MFG equilibrium.

\subsubsection{Literature review}

We now detail MFGs and contextualize our results through: stationary MFGs, MFGs with singular controls, and linear programming analysis for MFGs.

The most extensively studied MFG models are those on finite horizons with regular controls. In finite-horizon MFGs, the empirical distribution of players under Nash equilibrium (in a finite-player game) is approximated using a flow of measures, with the initial distribution of players predetermined. However, in the case of stationary problems, the flow of measures is replaced by a stationary measure. When the process is initialized with a stationary distribution, its distribution remains constant over time. Consequently, the initial distribution, which is stationary, becomes part of the solution. Stationary MFGs have been investigated in various contexts. For instance, in \cite{Adlakha}, discrete-time stationary MFGs with infinite horizons are explored. In \cite{diogo15}, the PDE approach is employed to examine stationary MFG solutions with quadratic Hamiltonians and congestion effects. The papers \cite{car-por, CZ2022} study the ergodicity of MFG systems of differential equations.  Additionally, \cite{Anahtarci2019ValueIA} proposes a value-iteration algorithm for studying stationary MFGs and convergence to equilibrium. In \cite{Neumann20}, stationary MFGs with finite state and action spaces, represented by time-inhomogeneous Markov chains, are considered under discounted reward criteria. The paper \cite{aydin2023robustness} extends the investigation of discrete-time stationary MFGs to incorporate model uncertainty settings. Lastly, in \cite{cannerozzi2024cooperation}, the authors study an ergodic MFG and corresponding $N$-player game with singular control, in one-dimension, under both both cooperative and competitive settings. They explicitly construct the solution to the mean-field problem, the Nash equilibria and the coarse-correlated equilibria, and further provide approximation results for mean-field solution to the $N$-player game.

One compelling application of stationary MFGs is their utilization in models with ergodic (long-horizon) costs. In \cite{Feleqi13}, for instance, the author explores MFGs using the Hamiltonian--Jacobian--Bellman equation approach, particularly in scenarios where dynamics occur over a compact space. Similarly, \cite{Bardi14} delves into ergodic MFGs and their correlation with $N-$player games, especially when the game takes on a linear-quadratic form. In \cite{Ari17}, the authors extend the study to multi-dimensional MFGs with ergodic costs and regular controls, demonstrating existence using both the Hamiltonian--Jacobian--Bellman approach and Kakutani--Glicksberg--Fan fixed-point theorem.

Another rationale for investigating stationary MFGs stems from the prevalence of economic models focused on stationary equilibria. As noted in \cite{Neumann20}, identifying stationary equilibria is comparatively less intricate, especially when applying fixed-point theorems.

Since our model includes singular controls, we survey several MFG models that incorporate singular controls, examining both finite- and infinite-horizon scenarios. In \cite{Fu21}, the authors investigate a finite horizon MFG with singular controls, demonstrating the existence of relaxed MFG solutions and exploring their approximation by regularly controlled MFG solutions using the M1 topology (see \cite{coh2019b} for more details on this topology and the relationship to singular control problems). In \cite{NMFG19}, the fuel-followers problem in the finite horizon is explored, establishing equilibrium existence for both $N$-player games and MFGs, with the MFG solution characterized as an $\eps$-Nash equilibrium for the $N$-player game. The paper \cite{MFGrever} focuses on discounted MFGs with singular control arising from partially reversible problems, analyzing the sensitivity of MFG solutions to model parameters and demonstrating $\eps$-optimality in the $N$-player game. Additionally, in \cite{finite-fuel}, discounted MFGs with singular control are examined within the context of finite-fuel capacity expansion, with connections drawn to $N$-player Nash equilibria through the use of Skorohod mapping. The paper \cite{aid2023stationary} investigates singular controlled MFGs originating from a commodity market with irreversible investment, establishing the existence and uniqueness of stationary MFG solutions under discounted criteria. In \cite{unifying}, submodular MFGs are explored, with sufficient conditions provided for the existence of MFG equilibria. The paper \cite{HYCao} delves into MFGs with singular control and ergodic cost in one dimension, exploring their relationship with discounted criteria and $N$-player games. Meanwhile, \cite{Fu2023} investigates MFGs with singular control under a finite horizon, utilizing continuous control as an approximation to singular control. Finally, \cite{dianetti2023ergodic} focuses on MFGs with singular control, incorporating regime-switching dynamics.

We now position our contribution in terms of modeling, results, and techniques within the broader context of existing research on 
MFGs with linear programming formulations. The paper \cite{Roxana2020} 
considers MFGs with optimal stopping. In \cite{Roxana2021}, 
the authors incorporate regular controls into the settings. Their methodology is to establish connections between the single agent game and its linear programming form; then use tightness arguments to show the existence of optimum in the linear programming form. Finally, they use Kakutani--Glicksberg--Fan fixed point theorem for the MFGs in the linear programming form to derive the existence of MFG equilibria. Our methodology is similar, yet the technical parts are different as we apply it to the stationary control problem and with singular controls. Also, these papers require the coefficients to be bounded or/and Lipschitz continuous, and so do not cover the population dynamics diffusion model in biology that is used in the harvesting model in population dynamics in Example \ref{lok-vol}, and their models consider only one-dimensional dynamics.

\subsection{Preliminaries and notation}
We use the following notation. The sets of {\it natural} and {\it real} numbers are respectively denoted by $\N$ and $\R$. For positive integers $a<b$, we use $\lbrs{a} :=\{1,\ldots,a\}$, and use $\lbrs{a,b} :=\{a,\dots,b\}$. For any $m\in\N$, and $a,b\in\R^m$, $a\cdot b$ denotes the {\it dot product} between $a$ and $b$, and $|a|=(a\cdot a)^{1/2}$ is the Euclidean norm. 
For $a,b\in\R$, set $a\vee b:=\max\{a,b\}$ and $a\wedge b:=\min\{a,b\}$. The interval $[0,\iy)$ is denoted by $\R_+$. 
Also, we will denote by $\calC^2_b(S)$ 
the set of functions that are twice continuously differentiable, and first derivative vanishes outside
a compact domain.
For any event $A$, $\one_A$ is the indicator of the event $A$, that is, $\one_A=1$ if $A$ holds and $0$ otherwise. We use the convention that the infimum of the empty set is $\iy$. We let $\calM(S)$ (reps., $\calP(S)$) is the space of finite Borel measures (resp., probability measures) on $S$. For a random variable $\xi$, let $\calL(\xi)$ denote its law or distribution.

\subsection{Organization}
The remainder of the paper is structured as follows: \red{Section \ref{min:sec} delineates the singular control problem for the minimization problem, and the proof is in Section \ref{sec3:min}, where we establishes the existence of optimal control within this framework. Section \ref{max:sec} on the other hand, describes the problem for the maximization setting, and its proof is in Section \ref{sec3:max}. Finally, Section \ref{sec4} presents the results for the MFG.}

\section{Case (\lowercase{i}): The queueing-inspired problem}\label{min:sec}
In this section, we analyze a problem inspired by the queueing model, which is a cost-minimizing problem. We shall refer to it as \textit{Case $(i)$}. Throughout the paper, we fix the dimensions $d,d_1, d_2\in\N$ associated with the underlying processes as we detailed in the introduction in the strict formulation and below in the relaxed framework. 
The process $X$ lives in $\X$, detailed below in  Assumption \ref{ass:21:min}; $W$ is a $d_1$-dimensional Wiener process; and $F$ is a (component-wise) non-decreasing right-continuous with left-limits (RCLL) process, whose increments live in $\FF=\R_+^{d_2}$. Before introducing the singular control problem, let us first introduce the underlying uncontrolled state process.

\subsection{The uncontrolled state process}
The following assumption primarily ensures the existence of a stationary martingale problem solution in the absence of control, restricts the growth of the coefficients in the dynamics, and guarantees the finite moment for the stationary distribution in the absence of control.

\begin{assumption}\label{ass:21:min}
    \begin{enumerate} 
\item[($A_1$)] There exists a filtered probability space  satisfying the usual conditions, $(\Omega,\calF,(\calF_t)_{t\in[0,1]},\PP)$, supporting $X$ and $\Phi$, a family of random measures on $\X\times\mathbb F\times[0,1]$. There also exists a probability measure $\nu\in \calP(\X)$, such that, for any $t\in[0,1]$, $\calL(X_t) = \nu$ and for any $h \in \calC_b^2(\X)$, the process
\begin{align}\notag
\calM^h_t:=h(X_t)-\int_0^t(\calA h)(X_s)ds-\sum_{i=1}^{d_2}\int_{\X\times\FF\times[0,t]}(\calB_i h)(x,z)\Phi_i(dx,dz,ds),
\end{align}
is an $\calF_t$-martingale, where $\calA$ and $\calB$ are defined in the introduction. Further, $\Phi$ has stationary increments, meaning for $t\in[0,1]$, $i = 1,\dots,n$, and $a_i<b_i$, the distribution of $(\Phi(H_1\times(a_1+t,b_1+t)),\dots,\Phi(H_n\times(a_n+t,b_n+t)))$ does not depend on $t$, \red{where $H_i\in\X\times\mathbb F$}. $\Phi$ is RCLL and $\calF_t$-adapted, meaning for each $i\in\lbrs{d_2}$, and measurable $A\in\X$ and $B \in \mathbb F$, the process $\Phi_i(\cdot) := \Phi_i(A\times B \times[0,\cdot])$ is RCLL and adapted to $\calF_t$.
\item[($A_2$)] The function $\sigma$ is locally bounded, and has at most linear growth, i.e.,
\begin{align*}
    \limsup_{|x|\to\iy}\frac{|\sig(x)|}{|x|} < \iy.
\end{align*}
\item[($A_3)$]
$\X :=\R^d$.
There exists a constant $p' >1$ such that
\begin{align*}
    \int_\X |x|^{p'}\nu(dx) < \iy,
\end{align*}
where $\nu$ is the stationary distribution from $(A_1)$.
\end{enumerate}
\end{assumption}

Notice that our Assumption $(A_3)$, is similar in role to the Assumption \cite[(9), (11)]{menaldi-robin}, both guaranteeing that the cost is finite in the absence of control.

In the strict formulation, 
the uncontrolled state process $X$ is given by:
\begin{align}\label{eq:state_process}
X_t=X_0 + \int_0^t \beta(X_s)ds+\int_0^t\sigma(X_s)d W_s  + \int_0^t \phi(X_s)d\Fs.
\end{align}

We now provide example of the uncontrolled state process, which are actually in the strict formulation provided in \eqref{eq:state_process}. Here, we may use It\^o's lemma to $h(X_t)$ to restore Assumption \ref{ass:21:min} $(A_1)$.

\begin{example}[Ornstein--Uhlenbeck process]
\label{OU-pro}
Consider the Ornstein--Uhlenbeck process in $\R^d$, where we have:
    \begin{align*}
        dX_t = -\theta X_t dt +\sigma dW_t,
    \end{align*}
    $\theta$ and $\sigma$ are both $d\times d$ matrix with positive entries and $F = 0$. It is well known that the Ornstein--Uhlenbeck process admits a stationary solution with exponentially decaying density, and therefore it satisfies Assumption \ref{ass:21:min} with $\Phi \equiv 0$.   
\end{example}

The above examples have $\phi\equiv 0$ (and $\Phi$), we now provide examples where the function $\phi$ (and $\Phi$) is not constantly $0$, and in particular, the process has a forced reflection at the boundary $0$.

\begin{example}[
Reflected Ornstein--Uhlenbeck]\label{rem:ref}
    Consider the example of a one-dimensional Ornstein--Uhlenbeck process with reflection at $0$. 
    The rigorous definition utilizes the Skorohoud map. Specifically, the dynamics of $X$ are given by 
    \begin{align*}
        X_t = X_0 - \int_0^t \theta X_s ds + \int_0^t \sig dW_s +\Ft.
    \end{align*}
In this example, we have $d=d_1=d_2=1$, $\phi\equiv 1$, and $F$ itself is the reflection part, reflecting the process at 0 upwards. Namely, 
the process $F$ is the minimal non-decreasing process making $X_t \geq 0$ for $t \in[0,1]$, and satisfies
    \begin{align*}
        \int_0^1 \one_{\{X_t >0\}} d\Ft = 0.
    \end{align*}
    We now show that indeed, this model satisfies our Assumption \ref{ass:21:min}. For  $(A_1)$, it follows by \cite[Proposition 1]{ref-OU} that the reflected Ornstein--Uhlenbeck process admits a stationary distribution, whose density is
    \begin{align*}
        \nu(dx) = 2\sqrt{\frac{2\theta}{\sig^2}}\varphi\Bigg(\sqrt{\frac{2\theta}{\sig^2}x}\Bigg),\qquad x\ge 0,
    \end{align*}
    where $\varphi(\cdot)$ is the density function of standard normal. For $(A_2)$, first notice that the constant diffusion coefficient $\sig$ is sub-linear, and for $(A_3(ii))$, since the density of $\nu(dx)$ decays exponentially, we have for any $p' >1$,
    \begin{align*}
        \int_\X |x|^{p'}\nu(dx) < \iy.
    \end{align*}
    Namely, Assumption \ref{ass:21:min} holds.
\end{example}

The following example includes a fundamental process in queueing theory, representing a diffusion limit of a scaled queueing system under heavy traffic conditions.
\begin{example}[
Multidimensional reflected Brownian motion]\label{rem:rbm}
    Consider the state dynamics
    \begin{align*}
        dX_t =B dt + \Sigma dW_t + Rd\Ft,
    \end{align*}
    where $B\in\R^d$, $\Sigma\in\R^{d\times d_1}$, and $R\in\R^{d\times d}$ is the reflection matrix. Similarly to Example \ref{rem:ref}, $F$ is the $d_2=d$-dimensional reflection process, reflecting the process at 0 upwards component-wise in a minimal way. Harrison and Williams \cite[4.(5), 4.(6)]{RBM_Harrison} show that under certain conditions on the coefficients $B, \Sigma$ and $R$, e.g., the process has a stationary distribution, whose density is separable, and decays exponentially.  

More generally, Harrison and Williams demonstrate in another study that the multidimensional reflected Brownian motion model can be expanded to encompass a smooth and bounded domain or a convex polyhedral. Under appropriate conditions as outlined in \cite[Theorem 2.1, Theorem 6.1]{MRBM-sta}, the process also exhibits an exponentially decaying stationary distribution.

Therefore, similarly to Example \ref{rem:ref}, Assumption \ref{ass:21:min} holds.
\end{example}

\subsection{The controlled state process}

Subsequently, we introduce the relaxed controlled process and the set of admissible controls. For this, we set up another dimension parameter $d_3\in\N$. We set $\G := \R^{d_3}_+$ to be the non-negative orthant of the $d_3$ dimensional space. The relaxed controls will be random measures whose increments belong to $\G$, and satisfy further conditions detailed below.

We now introduce the definition of {\it admissible relaxed singular controls}.
\begin{definition}\label{def21:min}.
An {\rm admissible relaxed singular control} is a tuple $$(\Omega,\calF,(\calF_t)_{t\in[0,1]},\PP,X,\nu,\{\Phi_i\}_{i\in\lbrs{d_2}},\{\Gamma_i\}_{i\in\lbrs{d_2+1,d_2+d_3}}),$$ 
 where $(\Omega,\calF,(\calF_t)_{t\in[0,1]},\PP)$ is a filtered probability space  satisfying the usual conditions and supporting the processes $X$, $\Gamma$ and $\Phi$, that satisfy the following conditions:
\begin{enumerate}[(a)]
\item 
$\Phi$ and $\Gamma$ are two families of random measures on $\X\times\FF\times[0,1]$ and $\X\times\G\times[0,1]$, respectively, RCLL, $\calF_t$-adapted, and with stationary increments, as defined in ($A_1$).
\item $X=(X_t)_{t\in[0,1]}$ is the state process, satisfying for any $t\in[0,1]$, any $h \in \calC_b^2(\X)$, $X_t\in\X$ and the process
\begin{align}\label{mar:cha:ctr}
\begin{split}
\calM^{\Gamma, h}_t:=h(X_t)-\int_0^t(\calA h)(X_s)ds-\sum_{i=1}^{d_2}\int_{\X\times\FF\times[0,t]}(\calB_i h)(x,z)\Phi_i(dx,dz,ds)\\-\sum_{i=d_2+1}^{d_2+d_3}\int_{\X\times\G\times[0,t]}(\calB_i h)(x,y)\Gamma_i(dx,dy,ds).
\end{split}
\end{align}
is an $\calF_t$-martingale.
\item The controlled process $X$ admits the stationary distribution $\nu\in \calP({\X})$; namely, for any $t\in[0,1]$, $\calL(X_t) = \nu$.
\end{enumerate}
\end{definition}

We denote the set of all admissible relaxed singular controls (in short, admissible controls) in Case (i) by $\calG_{(i)}$. We often abuse notation and refer to $\Gamma$
as the control instead of the tuple. 
On occasions, we use the notation $X^\Gamma$ and $\nu^\Gamma$, 
when we wish to emphasize that the dynamics of $X$ and the stationary distribution $\nu$ are driven by the control $\Gamma$.
Throughout the paper we consider only the case $X_0\sim\nu$ and use $\PP(A)$ and $\EE$ for the probability of the event $A$ and the expectation, respectively, assuming $X_0\sim\nu$.

Now, we provide a few examples where we actually have (strict) singular controls, defined as follows.

\begin{definition}\label{def:strict:min}
An {\rm admissible strict singular control} is a tuple $$(\Omega,\calF,(\calF_t)_{t\in[0,1]},\PP,X,W,F,G,\nu),$$ 
 where $(\Omega,\calF,(\calF_t)_{t\in[0,1]},\PP)$ is a filtered probability space  satisfying the usual conditions and supporting the processes $X,W,F,$ and $G$ that satisfy the following conditions:
\begin{enumerate}[(a)]
\item $W$ is a $d_1$-dimensional, $\calF_t$-measurable Wiener process and $F$ is a non-decreasing RCLL, $\calF_t$-adapted process, whose increments belong to $\FF$, and $F$ has stationary increments, meaning for $t\in[0,1]$, $i = 1,\dots,n$, and $a_i<b_i$, the distribution of $(F_{b_1+t}-F_{a_1+t},\dots,F_{b_n+t}-F_{a_n+t})$ does not depend on $t$;
\item $G=(\Gt)_{t\in[0,1]}$ is a non-decreasing RCLL, $\calF_t$-adapted process whose increments belong to $\G$, and $G$ has stationary increments, defined similar as above. 
\item $X=(X_t)_{t\in[0,1]}$ is the state process, satisfying for any $t\in[0,1]$, $X_t\in\X$ and 
\begin{align}\notag
X_t=X_0+\int_0^t\beta(X_s)ds+\int_0^t\sigma(X_s)d W_s+ \int_0^t \phi(X_s)d\Fs-\int_0^t\gamma(X_s)d\Gs, \quad X_0 \sim \nu.
\end{align} 
\item The controlled process $X$ admits the stationary distribution $\nu\in \calP({\X})$; namely, for any $t\in[0,1]$, $\calL(X_t) = \nu$. 
\end{enumerate}
\end{definition}

We now show that any strict controls give relaxed controls.
\begin{remark}
    Let $G$ be an admissible strict control. Define 
    \begin{align}\label{Gamma}
\Gamma^G_i(A\times B\times [0,t]):=\int_0^t\one_{A\times B}(X^G_{s-},\delta \Gs)d\Gis, \qquad i\in\lbrs{d_2+1,d_2+d_3},
\end{align}
and
\begin{align}\notag
\Phi^G_i(A\times B\times [0,t]):=\int_0^t\one_{A\times B}(X^G_{s-},\delta \Fs)d\Fis, \qquad i\in\lbrs{d_2},
\end{align}
\red{where $\delta G_s:=G_s-G_{s-}$.} 
Then, Conditions (a), (c) of Definition \ref{def21:min} are satisfied automatically. For Condition (b), applying It\^o's lemma to $h(X_t)$ gives the result.
\end{remark}

\begin{example}[Trivial control is admissible]\label{V_int:min}
    We now show that the trivial control is admissible, namely, 
     $\Gamma(\X\times\G\times[0,1]) = 0 \in\calG_{(i)}$. Obviously, it is adapted and has stationary increments.
     Assumption \ref{ass:21:min} $(A_1)$ guarantees Conditions (b) and (c) of Definition \ref{def21:min}.
\end{example}

\begin{remark}[
Up and down singular control]\label{Up-down}
Notice in Case (i), we are not restricting the direction of singular control, and in particular, we allow singular control in both directions. Such problems are considered, e.g., in \cite{Haussmann1}, where the authors have $\nu^1, \nu^2$ representing in different directions. In \cite{Instantaneous}, where the authors study one-dimensional storage problems, with both up and down singular controls, $R$ and $L$. In \cite{Cao-Guo-Lee} the authors examine one-dimensional $N$-player games and their relations to MFGs, with singular control that in each dimension of the dynamic, $i\in\lbrs{d}$, it has two opposite directions: plus (up), $+\nu^1_i$, and minus (down), $-\nu^2_i$. In the paper \cite{kurtz2017linear}, there is only one singular component, and while they do allow the singular control to have different actions on different dimensions of the state process $X$, it does not incorporate the above-mentioned models where different singular controls affect the same component of $X$ in different ways simultaneously. \red{Another example is \cite{Fu2023}, where the author establishes the existence of equilibrium solutions for a class of mean field games with singular controls, which are permitted to act in opposite directions.}
\end{remark}

\subsection{The optimization problem and the main single-agent result}

We now introduce the control problem for Case (i).
The {\it cost function} associated with the admissible control $\Gamma \in\calG_{(i)}$ is given by
\begin{align}\label{210a:min}
J(\Gamma):=
\int_{\X}\ell(x)\nu^\Gamma(dx) +\int_{\X\times\FF}f(x)\cdot\la^\Gamma(dx,dz)+\int_{\X\times \G}g(x)\cdot\mu^\Gamma(dx,dy),
\end{align}
where, we have 
\begin{align*}
&\la^\Gamma_i \in \calM({\X}\times\FF),\qquad \la_i^\Gamma(A\times B) = \E[\Phi_i(A\times B \times [0,1))], \qquad i\in\lbrs{d_2}, \\
&\mu^\Gamma_i \in \calM({\X}\times\G),\qquad  \mu_i^\Gamma(A\times B) = \E[\Gamma_i(A\times B \times [0,1))],\qquad i\in\lbrs{d_2+1,d_2+d_3}. 
\end{align*}
By $\int_{\X\times \FF}f(x)\cdot\la^\Gamma(dx,dz)$ we mean $\sum_{i=1}^{d_2} \int_{\X\times \FF}f_i(x)\la_i^\Gamma(dx,dz)$, and similarly for $g\cdot\mu$. The functions $f: {\X} \to \R^{d_2}_+$, $g: {\X} \to \R^{d_3}_+$, and $\ell:{\X}\to \R$ are measurable functions, satisfying further properties given in Assumption \ref{ass:22:min} below. 
The associated {\it values} in Cases (i) is given by
\begin{align*}
U_{(i)}:=
\inf_{\Gamma\in\calG_{(i)}}J(\Gamma)
\end{align*}
An admissible relaxed control $\Gamma$ is called an {\it optimal control} for Case (i) if its associated reward attains the value; that is, $J(\Gamma)=U_{(i)}$.

In Case (i), the control $\Gamma$ taken by decision maker will affect the state process $X$, and therefore indirectly affect $\Phi$, for example, in case of reflection. We put costs $f$ and $g$ to both the measures $\mu$ and $\la$, and therefore, the decision maker needs to decide whether to exercise some control $\Gamma$, and take cost $g\cdot\mu$, or wait to trigger the intrinsic singular component $\Phi$, and take the $f\cdot\la$ cost.

We make the following assumption on the cost components.
\begin{assumption}\label{ass:22:min}
\begin{enumerate}
\item[($A_4$)] The function $\ell$ is continuous. The function $g$ is continuous and non-negative.

   \item[($A_5$)]  There exist positive constants $c_{l,1},c_{l,2}$, and $c_{l,3}$ and $ p\in[1,p']$ such that for any $x\in{\X}$,
\begin{align}\notag
c_{l,1}|x|^p-c_{l,2}\le \ell(x)\le c_{l,3}(|x|^p+1).
\end{align}
Further,
\begin{align*}
&\lim_{|x|\to\iy} f_i(x) = \iy, \quad \forall i\in\lbrs{d_2},\\
    &\lim_{|x|\to\iy} g_i(x) = \iy, \quad \forall i\in\lbrs{d_2+1,d_2+d_3}.
\end{align*}
And,
\begin{align}\label{g:cost:finite}
    \int_{\X\times\FF}f(x)\cdot\la^0(dx,dz) < \iy,
\end{align}
where $\la^0$ is the collection of measures associated with the  trivial control $\Gamma \equiv 0$.
\end{enumerate}

\end{assumption}

\begin{remark}\label{queueing-model-cost}
\red{One may compare the Assumption above with \cite[(5)]{bud-ros2006}, and \cite[Assumption 2.2 $(A_6(a))$]{coh2019b}. In Case (i), we assume in $(A_5)$ that the cost function $\ell$ is bounded below a function that diverges at infinity. This implies that the cost function penalizes unstable behavior of the controlled process and is referred to as a {\it near-monotone} condition. The condition \eqref{g:cost:finite}, guarantees that the model has a finite intrinsic cost in the absence of control. We have similar assumptions on the cost functions as in \cite{Haussmann1,Haussmann2,Haussmann3}, but we do require all the cost components to diverge at infinity. In \cite{kurtz2017linear}, the authors require either the singular cost or their budget constraint function to be bounded away from zero. However, in fact, as explained in the Remark below, we may want to set the singular cost $f$ to be $0$ within a compact set, in certain queueing models.}

\red{In queueing models, the running cost $\ell$ stands for the holding cost of customers in the queue, which increases with the number of customers. The function $g$ may stand for the rejection cost, and the cost $f$, which is associated with the intrinsic reflection, may stand for an idleness penalty (due to losing jobs for not assigning them to a server). Note that our assumption above cares about the system's behavior near infinity. Recalling the reflected Brownian motion from \cite{MRBM-sta} with bounded domain $D$, we can set $f \equiv 0$ inside $D$. Then it eliminates the cost to the intrinsic reflection because $f$ and $\la$ have non-intersecting supports. This is because we may set up $f$ in an arbitrary way outside the domain $D$ such that the liminf condition holds and in addition the exponential ergodicity of the reflected Brownian motion guarantees \eqref{g:cost:finite}.}
\end{remark}

We now present the main theorem for the single-agent model.

\begin{theorem}\label{thm:main:min}
Under Assumptions \ref{ass:21:min} and \ref{ass:22:min}, the stationary relaxed singular control problem in Case (i) admits optimal relaxed admissible controls, and the value is finite.
\end{theorem}

\red{Notice that we establish the existence of relaxed controls, rather than strict.}

\section{Case (ii): The harvesting-inspired problem}\label{max:sec}
In this section, we turn to the analysis a problem inspired by the harvesting model in mathematical biology, which is formulated as a reward-maximizing problem. We shall refer to it as \textit{Case $(ii)$}. Similarly to the previous case, let us first introduce the underlying uncontrolled state process.

\subsection{The uncontrolled state process}

The following assumption primarily ensures the existence of a stationary martingale problem solution in the absence of control, restricts the growth of the coefficients in the dynamics, and ensures the existence of a Lyapunov function given below in Remark \ref{remark21}. This Lyapunov function is instrumental in providing stability to the dynamics, a crucial element for establishing various bounded results in Section \ref{sec3:max}, and at the same time guarantees the finite moment as well. The finite moment property is crucial in showing the finiteness of the optimum when we consider the control problem.

\begin{assumption}\label{ass:21:max}
    \begin{enumerate} 
\item[($B_1$)] There exists a filtered probability space satisfying the usual conditions, $(\Omega,\calF,(\calF_t)_{t\in[0,1]},\PP)$, supporting $X$. 
There also exists a probability measure $\nu\in \calP(\X)$, such that, for any $t\in[0,1]$, $\calL(X_t) = \nu$ and for any $h \in \calC_b^2(\X)$, the process
\begin{align}\notag
\calM^h_t:=h(X_t)-\int_0^t(\calA h)(X_s)ds,
\end{align}
is an $\calF_t$-martingale, where $\calA$ is defined in the introduction. 
\item[($B_2$)] Same as Assumption \ref{ass:21:min} ($A_2$).

\item[($B_3$)]
$\X :=\R^d_+$. The function $\beta$ is locally bounded, component-wise bounded from above, and satisfies 
\begin{align*}
    \lim_{|x|\to \iy} \beta(x) =-\iy,
\end{align*}
and for some $r \geq 1$ and a positive constant  $\underline{b} > 0$,
\begin{align*}
    \liminf_{|x|\to \iy} \frac{|\beta(x)|}{|x|^r} = \underline{b}, \qquad \lim_{|x|\to \iy} \frac{|\beta(x)|}{|\sig(x)|} = \iy.
\end{align*}

\end{enumerate}
\end{assumption}

We now compare our assumptions with some works of literature. In \cite{Haussmann1,Haussmann2, Haussmann3}, the coefficients $\beta, \sig$ are assumed to be bounded and continuous. In \cite{Roxana2020, Roxana2021}, assumptions exist that the coefficients are bounded or/and Lipschitz in the state variable. We do not impose such constraints here, and therefore our model can cover the Lokta--Volterra model from biology, given in Example \ref{lok-vol}. 

\color{black}
\begin{remark}\label{rem:diffusion}
Note that unlike in the previous case, here we assume $\Phi\equiv 0$, i.e., the uncontrolled process has no intrinsic singular part, as this does not show up in Assumption \ref{ass:21:max} $(B_1)$. In particular, we will take $d_2 = 0$ in Case (ii). This implies that the uncontrolled process satisfies a standard martingale problem. Since there is an equivalence between a weak solution to the stochastic differential equation and the solution to the martingale problem (see, e.g., \cite[Proposition 5.4.6]{KSbook}), we thus have a strict formulation of the uncontrolled process, where $X$ is given by:
\begin{align}\label{eq:state_process:max}
X_t=X_0 + \int_0^t \beta(X_s)ds+\int_0^t\sigma(X_s)d W_s .
\end{align}
This form is being utilized in Example \ref{V_int:max}, where we show that the trivial control is admissible in this case. Unlike the previous case, however, the admissibility of the controls requires an additional condition that reflects the stability of the process. Thus, we employ the stochastic differential equation form above, as stability is more tractable in this framework.

However, this is not the main reason we assume $\Phi\equiv 0$. This feature is crucial when calculating the occupation measure $\mu$ in this case, as the extra occupation measures introduced by 
$\Phi$ would violate the technical bound in Proposition \ref{measure:compact:max}. 
\end{remark}
\color{black}

The following remark provides a family of \textit{Lyapunov functions} for Case (ii), which are used in Proposition \ref{measure:compact:max} in order to show that the
stationary distribution of the controlled process and the occupation measure exhibit vanishing weights at infinity.
\begin{remark}[Lyapunov functions for Case (ii)]\label{remark21}
Fix $V(x) := \sum_{i=1}^d x_i^{K_r}$, where $K_r \geq 2r$ is an arbitrary constant. The above assumptions for Case (ii) guarantee that $V$ is a Lyapunov function, 
i.e., it satisfies the condition: 
\begin{align}\label{eq:lyapunov2}
    \calA V(x) \leq k_1-k_2V(x),\qquad \forall x\in \X.
\end{align}
Its full power will be used in Remark \ref{rem:23}, and then serves us in Sections \ref{sub32:max} and \ref{sec:34:max}. Indeed, notice that since $\sig$ has at most linear growth, and $\beta(x)\sim |x|^r$, therefore, we have
\begin{align*}
    \calA V(x) &= \frac 12 \sum_{i,j=1}^d a_{ij}(x)\frac{\partial^2 V(x)}{\partial x_i\partial x_j}+\sum_{i=1}^d\beta_i(x)\frac{\partial V(x)}{\partial x_i}\\&= \sum_{i=1}^d \sigma_{ii}^2(x)K_r(K_r-1)x_i^{K_r-2} + \sum_{i=1}^d\beta_i(x)K_rx_i^{K_r-1}\\
    &=
    o(|x|^{K_r-1}|\beta(x)|) - \Omega(|x|^{K_r -1}|\beta(x)|) \\&
    = - \Omega(|x|^{K_r}),
\end{align*}
\red{where we use the little $o$ and big Omega notation.}\footnote{\red{The \textit{big-Omega} notation, $f(x) = \Omega(g(x))$, describes the lower bound of a function. This means there exist constants $c > 0$ and $x_0$ such that for all $x \geq x_0$, $f(x) \geq c \cdot g(x)$. In contrast, the \textit{little-o} notation, $f(x) = o(g(x))$, indicates that $f(x)$ grows strictly slower than $g(x)$, meaning for any constant $c > 0$, there exists $x_0$ such that $f(x) < c \cdot g(x)$ for all $x \geq x_0$.
}} Specifically, the third and last equalities follow by Assumptions $(B_2)$ and $(B_3)$, using the fact that $r\ge 1$. 
Finally, since $V(x) \sim |x|^{K_r}$, there exist positive constants $k_1$ and $k_2$ satisfying \eqref{eq:lyapunov2}.
\end{remark}

We now provide examples of the uncontrolled state process in the strict formulation for Case (ii).

\begin{example}[\red{Population dynamics as an underlying model for the harvesting model}]
\label{lok-vol}
Assumption \ref{ass:21:max} holds for the competitive Lotka--Volterra model for population dynamics:
\begin{align*}
    dX_{i,t} = X_{i,t}\Big(\alpha_i - \sum_{j=1}^d\beta_{ji}X_{j,t}\Big)dt + \sig_i X_{i,t}dW^i_t, \quad i = 1,\dots,d,
\end{align*}
where $\alpha_i, \beta_{ji}, \sig_i > 0$ are constants. So far, Assumptions $(B_2)$, and $(B_3)$ hold with $r=2$. As for Assumption $(B_1)$, the process admits a stationary distribution under additional conditions that guarantee that all the species do not go extinct, i.e., $X_{i,t} > 0$ for all $i\in[d]$ and  $t\ge 0$, see e.g., \cite[Assumptions 1.1 and 1.2 and Theorem 1.1]{coexist}, and we take $\Phi \equiv 0$. \red{In this context, the singular control component added to the model is a nonincreasing process that represents the harvesting action. Together with this control, we refer to the model as the {\rm harvesting model in population dynamics}.}
\end{example}

\subsection{The controlled state process}

We will now introduce the relaxed controlled process and the set of admissible controls for Case (ii). The set-up is the same as in Case (i). We also need the following assumption on the measurable function $\gamma:\X\to\R_+^{d\times d_3}$ featured as the coefficient concerning the singular control. This assumption is relevant only for Case (ii). Essentially, it ensures that the singular control influences the dynamics in a non-degenerate and nonnegative manner, and will be important in Section \ref{sec:34:max}.

\begin{assumption}\label{k_matrix}($B_4$)
     The function $\gamma$ is bounded below in the following sense: there exists a constant ${\bar\gamma} > 0$, such that for all $x \in \X$ and any $1\le i\le d,1\le j\le d_3$, $\gamma_{i,j}(x) \geq {\bar\gamma}$. 
\end{assumption}

We now introduce the definition of {\it admissible relaxed singular controls}.
\begin{definition}\label{def21:max}
An {\rm admissible relaxed singular control} is a tuple $$(\Omega,\calF,(\calF_t)_{t\in[0,1]},\PP,X,\nu,\{\Gamma_i\}_{i\in\lbrs{1,d_3}}),$$ 
 where $(\Omega,\calF,(\calF_t)_{t\in[0,1]},\PP)$ is a filtered probability space  satisfying the usual conditions and supporting the processes $X$, $\Gamma$, that satisfy the conditions (a), (b), (c) in Definition \ref{def21:min} \red{with $\Phi \equiv 0$ and $d_2 = 0$}, and, further,
\begin{enumerate}[(d)]
\item Let $\Vfour$ be a Lyapunov function as in Remark \ref{remark21}, with $K_r = 4r$. Then, 
\begin{align*}
    \int_{\X} \calA \Vfour(x) \nu(dx) \geq 0.
\end{align*}
\end{enumerate}
\end{definition}

We denote the set of all admissible relaxed singular controls (in short, admissible controls) in Case (ii) by $\calG_{(ii)}$.

As mentioned in Remark \ref{remark21}, the following remark provides a uniform bound on the integral of the Lyapunov function under the stationary distribution of the controlled process. This bound will help us in Sections \ref{sub33:max} and \ref{sec:34:max}. It illustrates the relevance of Condition (d) above.
\begin{remark}\label{rem:23}
    Consider a control $\Gamma \in \calG_{(ii)}$. 
    \red{We apply the definition of the Lyapunov function \eqref{eq:lyapunov2} on $\Vfour$  along with Definition \ref{def21:max}(d), and we obtain the inequalities:
    \begin{align*}
        \int_{\X} k_1-k_2\Vfour(x) \nu^\Gamma(dx) \geq \int_{\X} \calA \Vfour(x) \nu^\Gamma(dx) \geq 0.
    \end{align*}
    Taking expectations on both the left- and right-hand sides, and reordering the equation, we obtain:}
    \begin{align}\label{Vfour}
       \EE[\Vfour(X_0^\Gamma)]= \int_\X\Vfour(x)\nu^\Gamma(dx) \leq \frac{k_1}{k_2} < \iy.
    \end{align}
   The bound mentioned above is uniform in $\Gamma$. Together with the observation that $\Vfour(x)\to\infty$ as $|x|\to \infty$, it implies that the stationary distribution of the process $X^\Gamma$ decays rapidly enough as $|x|  \to \infty$, at least faster than $\Vfour$ grows.

    Also, recall that $b$ and $\sig^2$ are of order $o(\Vfour)$, as $|x|\to\infty$. Hence, the above inequality is useful when we bound $\EE[\beta(X_t)]$ or $\EE[\sig^2(X_t)]$.  
\end{remark}

 Now, we define similarly the strict singular controls in Case (ii).
\begin{definition}\label{def:strict:max}
An {\rm admissible strict singular control} is a tuple $$(\Omega,\calF,(\calF_t)_{t\in[0,1]},\PP,X,W,G,\nu),$$ 
 where $(\Omega,\calF,(\calF_t)_{t\in[0,1]},\PP)$ is a filtered probability space  satisfying the usual conditions and supporting the processes $X,W,$ and $G$ that satisfy the conditions $(a), (b),(c),$ and $ (d)$ in Definition \ref{def:strict:min} with $F \equiv 0$, and
\begin{enumerate}[(e)]
\item Let $\Vfour$ be a Lyapunov function as in Remark \ref{remark21}, with $K_r = 4r$. Then, 
\begin{align*}
    \int_{\X} \calA \Vfour(x) \nu(dx) \geq 0.
\end{align*}
\end{enumerate}
\end{definition}

Now, we provide an example where we actually have (strict) singular controls.

\begin{example}[Trivial control is admissible]\label{V_int:max}
     We now show that the trivial control is admissible in Case (ii) as well. Conditions $(a),(b)$, and $(c)$ hold as shown in Example \ref{V_int:min}. For Condition $(d)$, notice that in Case (ii) we have the process being a diffusion as mentioned in Remark \ref{rem:diffusion}.  
     Let $X$ denote the process under strict control, i.e., the weak solution to stochastic differential equation \eqref{eq:state_process:max}. By the Lyapunov condition, we have (see e.g.  \cite[Lemma 2.5.5]{Ari_book})
    \begin{align*}
      \EE[\Vfour(X_0)]=  \int_{\X} \Vfour(x) \nu(dx) \leq \frac{k_1}{k_2} + e^{-k_2 t} \int_{\X} \Vfour(x) \nu(dx),
    \end{align*}
    which in turn implies 
    $\EE[\Vfour(X_0)] < \iy$. Applying It\^{o}'s lemma to the process $\Vfour(X_t)$, we get
    \begin{align*}
        \int_{\X} \calA \Vfour(x) \nu(dx) = 0.
    \end{align*}
    Altogether, we get Condition (d).
\end{example}

\begin{example}[
Downward reflecting control in dimension one is admissible]\label{1d-ref}
See e.g., \cite{alv-hen2019}. Consider the controlled process
    \begin{align*}
        X^G_t = X^G_0+\int_0^t\beta(X^\Gs)ds+\int_0^t\sigma(X^\Gs)d W_s -\Gt,
    \end{align*}
    with dimensions $d=d_1 =d_3= 1$, the matrix function $\gamma$ degenerates to the constant $1$, and $\phi\equiv 0$. Pick $c> 0$. The process $G$ is the minimal non-decreasing process making $X^G_t \in (0,c]$ for $t \in[0,1]$, and satisfies
    \begin{align*}
        \int_0^1 \one_{\{X^G_t <c\}} d\Gt = 0.
    \end{align*}
    Under the condition on the speed measure $\frak{m}(0,c)<\iy$, 
    where
    \begin{align*}
        &S'(x) = \exp\Big(-\int_0^x\frac{2\beta(y)}{\sig^2(y)}dy\Big),\qquad x\in (0,\iy),\\
        &\frak{m}(x,y) = \int_x^y \frac{2}{\sig^2(z)S'(z)}dz,\qquad y \in (x,\iy),
    \end{align*}
    the controlled process is ergodic and admits a stationary distribution on $(0,c)$, see \cite[II.~6.~36]{handBM}, which implies Condition (d). Condition (a) follows because both the dynamics and control depend on the current state, not the time. Conditions (b) and (c) are satisfied by the structure of $G$. For Condition (e), first notice that $\nu$ has a compact support on $(0,c)$, and $\Vfour$ is continuous. Therefore,
    \begin{align*}
        \int_\X\bar V(x)\nu(dx) < \iy.
    \end{align*}
    Applying It\^{o}'s lemma to the process $\Vfour(X)$ and taking expectation, we get
    \begin{align*}
        \int_{\X} \calA \Vfour(x) \nu(dx) + \int_{\X\times\G}\calB \Vfour(x,y)\mu(dx,dy) = 0.
    \end{align*}
    However, the second term is negative as the control is downward, and therefore
    \begin{align*}
        \int_{\X} \calA \Vfour(x) \nu(dx)\geq 0.
    \end{align*}
\end{example}

\subsection{The optimization problem and the main single-agent result}

We now introduce the control problem.
The {\it reward function} associated with the admissible control $\Gamma \in\calG_{(ii)}$ is given by
\begin{align}\label{210a:max}
J(\Gamma):=
\int_{\X}\ell(x)\nu^\Gamma(dx)+\int_{\X\times \G}g(x)\cdot\mu^\Gamma(dx,dy),
\end{align}
where, we have 
\begin{align*}
&\mu^\Gamma_i \in \calM({\X}\times\G),\qquad  \mu_i^\Gamma(A\times B) = \E[\Gamma_i(A\times B \times [0,1))],\qquad i\in\lbrs{1,d_3}. 
\end{align*}
By $\int_{\X\times \G}g(x)\cdot\mu^\Gamma(dx,dz)$ we mean $\sum_{i=1}^{d_3} \int_{\X\times \G}g_i(x)\mu_i^\Gamma(dx,dz)$. The functions $g: {\X} \to \R^{d_3}_+$, and $\ell:{\X}\to \R$ are measurable functions, satisfying further properties given in Assumption \ref{ass:22:max} below. 
The associated {\it values} in Cases (ii) is given by
\begin{align*}
U_{(ii)}:=
\sup_{\Gamma\in\calG_{(ii)}}J(\Gamma)
\end{align*}
An admissible relaxed control $\Gamma$ is called an {\it optimal control} for Case (ii) if its associated reward attains the value; that is, $J(\Gamma)=U_{(ii)}$.

We make the following assumption on the reward components.
\begin{assumption}\label{ass:22:max}
\begin{enumerate}
\item[($B_5$)] Same as Assumption \ref{ass:22:min} $(A_4)$: The function $\ell$ is continuous. The function $g$ is continuous and non-negative.

   \item[($B_6$)] The functions $\ell$ and $g$ are bounded.
\end{enumerate}

\end{assumption}

\begin{remark}\label{queueing-model-cost:max}
\red{Harvesting models are often formulated using population dynamics, as in Example \ref{lok-vol}, with a singular control component added to represent the harvesting action, which is a nonincreasing process. The reward is usually the yield, under which case, $\ell = 0$ and $f$ is constant. Both are bounded, hence, are a simple private case of our Assumption above. In \cite{Roxana2021}, where the maximization problem is considered, the authors assume similar conditions on the running reward, despite requiring it to be upper semicontinuous instead of continuous. We would like it to be continuous for the purpose of weak convergence of the reward, when we take the sequence of controls to limit. }

\red{From a technical standpoint, we require the functions $f$ and $g$ to be bounded, as the analysis in Subsection \ref{sec:35:max} relies on the weak convergence of the reward to the optimal one, which necessitates boundedness.}
\end{remark}

We now present the main theorem for the single-agent model.

\begin{theorem}\label{thm:main:max}
Under Assumptions \ref{ass:21:max}, \ref{k_matrix} and \ref{ass:22:max}, the stationary relaxed singular control problem in Case (ii) admits optimal relaxed admissible controls, and the value is finite.
\end{theorem}

\red{As in Theorem \ref{thm:main:min}, we establish the existence of relaxed controls, rather than strict.}

\section{Proof of Theorem \ref{thm:main:min}}\label{sec3:min}
This section is devoted to the proof of Theorem \ref{thm:main:min}, which is the main result of the single-agent control problem in Case (i). The proof is constructed in the following way.
In Section \ref{sub32:min}, we give a characterization of the control $\Gamma$ as the solution to a martingale problem. That is, we show that an admissible control $\Gamma$ is associated with a family of measures solving the martingale problem \eqref{mar_prm}.
On the other hand, given the martingale problem solution, we can reconstruct a relaxed admissible control $\Gamma$. With this equivalence, we have the control problem in the linear programming form. In Section \ref{sub33:min}, we study the control problem by first showing that the optimum over cost is finite. In Section \ref{sec:34:min}, we show the sequential compactness property of the measures associated with the admissible controls by tightness. The result is also necessary for the MFG part in Section \ref{sec4}. Finally, in Section \ref{sec:35:min}, we prove Theorem \ref{thm:main:min} by showing that the supremum can be obtained, and therefore we find the optimal solution to the control problems.

\subsection{Characterization for admissible controls}\label{sub32:min}
The next two lemmas characterize the relationship between admissible controls and stationary distribution using the generators $\calA$ and $\calB$, derived by a martingale problem. For this subsection, the characterization result is very general, i.e., applicable to a much larger set of problems, and therefore, the assumptions from the previous section, which are ``more restrictive" in a sense, are not used in their full power. Notice $d_2 \ne 0$ in Case (i), however as we show in the following remark, we can find a way to ``merge''.

\begin{remark}\label{rem:mu:la} Recall that $\FF = \R_+^{d_2}$ and $\G = \R_+^{d_3}$. For Case (i) however, we now show that without loss of generality, we may assume $d_2 = d_3$. To be clearer, let us introduce a new index $j$. Suppose first $d_2 < d_3$, then, we may enlarge the dimension of the coefficients matrix $\phi$ by adding column(s) of $\boldsymbol{0}\in\R^{d\times (d_3-d_2)}$ to be
\begin{align*}
    \hat\phi := [\phi\;,\; \boldsymbol{0}] \in \R^{d\times d_3}.
\end{align*}
Then, we will define,
\begin{eqnarray*}
    \hat\calB_j =
    \begin{cases}
    \calB_j, &j \in \lbrs{d_2},\\
     0, & j \in \lbrs{d_2+1,d_3},\\
     \calB_{j-d_3+d_2},
     &j \in\lbrs{d_3+1,2d_3},
    \end{cases}
\end{eqnarray*}
and similarly update the dimensions for $\hat \Phi_j$ and $\hat \la_j$. What we are doing here, is adding some new operators and measures that are constantly $\boldsymbol{0}$ to associate with the intrinsic singular component $\Phi$, and shift those associated with control $\Gamma$, from the original $\lbrs{d_2+1,d_2+d_3}$, to the new $\lbrs{d_3+1,2d_3}$. Doing so will not alter the problem as the newly introduced operators will be constantly $\boldsymbol{0}$. On the other hand, if $d_3 < d_2$, we enlarge the dimension of the coefficients matrix $\gamma$ similarly by adding column(s) of $\boldsymbol{0}\in\R^{d\times (d_2-d_3)}$, and for $i \in \lbrs{d_2+d_3+1,2d_2}$
, we introduce $\Gamma_i$, and introduce $h_i$ which explode at infinity as well. In this situation, while the $\Gamma_i$, and as a result $\mu_i$, for $i \in \lbrs{d_2+d_3+1,2d_2}$, could be non-trivial measure, there is no reason to take control and increase the cost, when the control cannot affect the state process since we are doing minimization. Therefore, the problem (optimal control and optimum) remains unchanged as well. 

As a result, without loss of generality, we will assume $d_2 = d_3$, and thus $\G = \FF$. We abuse notation, and in what follows, we still write index $i$ instead of $j$, and write $\calB$, $\Phi$ and $\la$ without hat. Further, note that $g$ and $f$ satisfy the same assumptions. Hence, in what follows, for conciseness, we will use $\mu_i$ for $i\in\lbrs{d_2}$ to denote the original $\la_i$. Further, we will also use $g_i$ for $i\in\lbrs{d_2}$ to denote the original $f_i$, particularly after Section \ref{sub32:min}, that is, we will use, whenever appropriate (in Case (i)),
\begin{align*}
    \int_{\X\times\G}g(x)\cdot\mu(dx,dy) = \sum_{i=1}^{d_2+d_3} g_i(x)\mu_i(dx,dy) = \sum_{i=1}^{d_2} f_i(x)\la_i(dx,dy)+\sum_{i=d_2+1}^{d_2+d_3} g_i(x)\mu_i(dx,dy).
\end{align*}

\end{remark}

\begin{lemma}\label{lem:exo}
Let $\Gamma$ be an admissible control with finite first moment, namely, $\EE[\Gamma(\X\times\G\times[0,1])]<\iy$ and for which the stationary distribution for state process $X^\Gamma$ is denoted by $\nu^\Gamma\in\calP({\X})$. Under Assumption 
$(A_2)$, there exist measures $ \{\mu_i^\Gamma\}_{i\in\lbrs{d_2+d_3}}\subset\calM({\X}\times\G)$, such that for all $h\in\calC^2_b({\X})$,  
\begin{align}\notag
\int_{{\X}}\calA h(x)\nu^\Gamma(dx)+\sum_{i=1}^{d_2+d_3}\int_{{\X}\times\G}(\calB_i h)(x,y)\mu_i^\Gamma(dx,dy)=0.
\end{align}
\end{lemma}

\begin{proof}
For the rest of the proof we denote $X=X^\Gamma$. Fix an arbitrary $h\in\calC^2_b({\X})$, for any $i\in\lbrs{d_2+1,d_2+d_3}$, define $\mu^\Gamma_i \in \calM({\X}\times\G)$ by $\mu_i^\Gamma(A\times B) = \E[\Gamma_i(A\times B \times [0,1])]$, and for $i\in\lbrs{d_2}$, define by $\mu_i^\Gamma(A\times B) = \E[\Phi_i(A\times B \times [0,1])]$. Then, since \eqref{mar:cha:ctr} is a martingale, taking expectation and using stationarity, we immediately get 
\begin{align}\notag
\int_{{\X}}\calA h(x)\nu^\Gamma(dx)+\sum_{i=1}^{d_2+d_3}\int_{{\X}\times\G}(\calB_i h)(x,y)\mu_i^\Gamma(dx,dy)=0.
\end{align}
\end{proof}

The following lemma is a multidimensional version of \cite[Theorem 1.7]{kur-sto}, in the sense that we consider a $d_3$-dimensional control process $\Gamma$, while in \cite{kur-sto}, the singular control process, in their notation, $F$, is $1$-dimensional, as implied by the real-valued equation (1.5) there (not to be confused with our intrinsic singular process $F_i$). It is the inverse of the previous lemma. 

\begin{lemma}\label{lem:exi}
Let Assumption 
$(A_2)$ be in force. Let $\nu\in\calP({\X})$ and $ \{\mu_i\}_{i\in\lbrs{d_2+d_3}}\subset\calM({\X}\times\G)$, be such that for all $h\in\calC^2_b({\X})$,  
\begin{align}\label{eq:lem32}
\int_{{\X}}\calA h(x)\nu(dx)+\sum_{i=1}^{d_2+d_3}\int_{{\X}\times\G}(\calB_i h)(x,y)\mu_i(dx,dy)=0.
\end{align}
Let $\mu^{\X}_i$ be the state-marginal of $\mu_i$ and let $\eta_i$ be the transition function from ${\X}$ to $\G$ such that $\mu_i(dx,dy)=\eta_i(x,dy)\mu^{\X}_i(dx)$. Then, there exist a process $X$ and  random measures $\{\hat\Gamma_i\}_{i\in\lbrs{d_2+d_3}}$ on ${\X}\times\R_+$, adapted to $\calF_t$, such that:
\begin{itemize}
\item $X$ is stationary and for any $t\in\R_+$ $X(t)$ has the distribution $\nu$;
\item $\hat\Gamma_i$ has stationary increments, $\hat\Gamma_i({\X}\times[0,t])$ is finite for any $t\in[0,1]$, and $\E[\hat\Gamma_i(\cdot\times[0,t])]=t\mu_i^{\X}(\cdot)$;\footnote{In \cite{kur-sto} there is a small typo and instead of $\E[\hat\Gamma(\cdot\times[0,t])]=t\mu_1(\cdot)$, it should be there $\E[\hat\Gamma_i(\cdot\times[0,t])]=t\mu_1^{\X}(\cdot)$.}
\item for any $h\in\calC^2_b({\X})$, the process $(N^h_t)_{t\in[0,1]}$ is an $\calF_t$-martingale, where
\begin{align}\notag
N^h_t:= h(X_t)-\int_0^t\calA h(X_s)ds-\sum_{i=1}^{d_2+d_3}\int_{{\X}\times[0,t]}\int_\G\calB_i h(x, y)\eta_i(x,dy)\hat\Gamma_i(dx,ds).
\end{align}
\end{itemize}
In particular, set $\Gamma_i(dx,dy,ds) := \eta_i(x,dy)\hat\Gamma_i(dx,ds)$ for $i \in \lbrs{d_2+1,d_2+d_3}$. Then, $\Gamma$ is an admissible control. As a result, for any $C\subseteq\X\times\G$,
\begin{align}\label{eq:mu-gamma}
\mu_i(C)=\EE[\Gamma_i(C\times[0,1])].
\end{align}
Similarly we set $\Phi_i$ for $i \in \lbrs{d_2}$, and obtain the $\la_i(C)=\EE[\Phi_i(C\times[0,1])]$.
\end{lemma}

\begin{proof} The proof involves reducing the problem to one dimension, specifically in terms of the singular control process, where the lemma is applicable by Theorem 1.7 in \cite{kur-sto}. Let us define $\tilde \mu \in \calM({\X}\times\G)$ by $\tilde \mu(A) = \sum_{i=1}^{d_2+d_3} \mu_i(A)$, and let $\tilde \mu^{\X}$ be the marginal of $\tilde \mu$ on ${\X}$, then for all $i \in \lbrs{d_2+d_3}$, $\mu_i$ (resp., $\mu_i^{\X}$) is absolutely continuous with respect to $\tilde \mu$ (resp., $\tilde \mu^{\X}$). Define for $i \in \lbrs{d_2+d_3}$ the Radon--Nikodym derivatives,
\begin{align*}
    \zeta_i = \frac{d\mu_i}{d\tilde \mu}\qquad\text{and}\qquad
    \tilde\zeta_i= \frac{d\mu_i^{{\X}}}{d\tilde \mu^{\X}}.
\end{align*}
Set also
\begin{align*}
    \tilde \calB h(x,y) = \sum_{i=1}^{d_2+d_3}\zeta_i(x,y)(\calB_i h)(x, y).
\end{align*}

Our next step is to show that the generators $\calA$ and $\tilde \calB$ satisfy the properties (i)--(v) of \cite[Condition 1.2]{kur-sto}. For (i), it is clear that the constant function $h(\cdot) \equiv 1 \in \calC^2_b({\X})$, and $\calA 1 = \tilde \calB 1 = 0$. For (ii), let $\psi_A, \psi_B = 1$, then, since $h \in \calC^2_b({\X})$ has bounded first derivatives, we may choose $a_f, b_f$ to be the maximum of $|\calA h|, |\tilde \calB h|$, respectively. Property (iii) follows since the set of functions that are equal to polynomials with rational coefficients within a compact domain, and equal to constant outside forms a countable dense subset of $\calC^2_b({\X})$. For (iv), notice that we do not consider regular controls, 
and $\calA$ and $\calB_i$ are generators by their construction. 
Finally, for (v), the set $\calC^2_b({\X})$ is closed under multiplication and separate points. Then, by the definitions of $\zeta_i$ and $\tilde\calB h$, and \eqref{eq:lem32}  
we have
\begin{align*}
    &\int_{{\X}}\calA h(x)\nu(dx)+\int_{{\X}\times\G}\tilde\calB h(x,y)\tilde\mu(dx,dy)\\ &\quad=\int_{{\X}}\calA h(x)\nu(dx)+\sum_{i=1}^{d_2+d_3}\int_{{\X}\times\G}(\calB_i h)(x,y)\zeta_i(x,y)\tilde\mu(dx,dy)\\
    &\quad=\int_{{\X}}\calA h(x)\nu(dx)+\sum_{i=1}^{d_2+d_3}\int_{{\X}\times\G}(\calB_i h)(x,y)\mu_i(dx,dy)
    \\
    &\quad=0.
\end{align*}
Recall that $\tilde\mu^{\X}$ is the state-marginal of $\tilde\mu$. Let $\tilde\eta$ be the transition function from ${\X}$ to $\G$ such that $\tilde\mu(dx,dy) = \tilde\eta(x,dy)\tilde\mu^{{\X}}(dx)$. At this point, we are ready to apply  \cite[Theorem 1.7]{kur-sto} to $\calA$ and $\tilde \calB$, thereby establishing the existence of a process $X$ and random measure $\tilde\Gamma$ on ${\X}\times[0,1]$, adapted to $\calF_t$, such that:
\begin{itemize}
\item $X$ is stationary and for any $t\in\R_+$, $X(t)$ has distribution $\nu$;
\item $\tilde\Gamma$ has stationary increments, $\tilde\Gamma({\X}\times[0,t])$ is finite for any $t\in[0,1]$, and 
\begin{align}\label{kur_gam}
\E[\tilde\Gamma(\cdot\times[0,t])]=t\tilde\mu^{\X}(\cdot);
\end{align}
\item for any $h\in\calC^2_b({\X})$, the process $(\tilde N^h_t)_{t\in[0,1]}$ is an $\calF_t$-martingale, where
\begin{align}\notag
\tilde N^h_t:= h(X_t)-\int_0^t\calA h(X_s)ds-\int_{{\X}\times[0,t]}\int_\G\tilde\calB h(x,y)\tilde\eta(x,dy)\tilde\Gamma(dx,ds).
\end{align}
\end{itemize}

Define $\{\hat\Gamma_i\}_{i\in\lbrs{d_2+d_3}}$ on ${\X}\times [0,1]$ by
\begin{align}\notag
    \hat\Gamma_i(A\times I) = \int_{A\times I}\tilde\zeta_i(x)\tilde\Gamma(dx,ds).
\end{align}
Notice that for all $i \in\lbrs{d_2+d_3}$ and measurable $A \subseteq {\X}$, $\mu_i^{\X}(A) \leq \tilde \mu^{\X}(A)$. Therefore, $\tilde\zeta_i(x)$ is bounded, thus the finiteness of $\hat\Gamma_i({\X}\times[0,t])$ follows, and we have
\begin{align*}\notag
    \EE[\hat\Gamma_i(\cdot\times[0,t])] = \EE\Big[\int_{\cdot\times[0,t]}\tilde\zeta_i(x)\tilde\Gamma(dx,ds)\Big]=\int_{\cdot\times[0,t]}\tilde\zeta_i(x)\tilde\mu^{\X}(dx)ds=t\mu_i^{\X}(\cdot),
\end{align*}
where the second equality comes from Fubini’s theorem.
Now, we show that 
\begin{align}\notag
    N^h_t:= h(X_t)-\int_0^t\calA h(X_s)ds-\sum_{i=1}^{d_2+d_3}\int_{{\X}\times[0,t]}\int_\G\calB_i h(x, y)\eta_i(x,dy)\hat\Gamma_i(dx,ds)
\end{align}
is an $\calF_t$-martingale, which follows once we show that for any $0 < s < t$,
\begin{align}\label{eq:multi_B}
    \sum_{i=1}^{d_2+d_3}\int_{{\X}\times[s,t]}\int_\G\calB_i h(x, y)\eta_i(x,dy)\hat\Gamma_i(dx,du)=
    \int_{{\X}\times[s,t]}\int_\G\tilde\calB h(x,y)\tilde\eta(x,dy)\tilde\Gamma(dx,du), \quad \PP\text{-a.s.}
\end{align}
To this end, recall the definition of $\eta_i$ from the statement of the lemma and the definitions of $\zeta_i$ and $\tilde \zeta_i$. Then, for any $i \in \lbrs{d_2+d_3}$, and any measurable $A\subseteq \X$, we have
\begin{align*}
    &\int_A\int_\G\calB_i h(x, y)\eta_i(x,dy)\tilde\zeta_i(x)\tilde\mu^{\X}(dx) \\
     &\quad=\int_{A\times\G}\calB_i h(x, y)\mu_i(dx,dy) \\
     &\quad=\int_A\int_\G\calB_i h(x, y)\zeta_i(x,y)\tilde\eta(x,dy)\tilde\mu^{\X}(dx).
\end{align*}
Therefore, as a function of $x$, 
\begin{align*}
    \int_\G\calB_i h(x, y)\eta_i(x,dy)\tilde\zeta_i(x) = \int_\G\calB_i h(x, y)\zeta_i(x,y)\tilde\eta(x,dy),\qquad \tilde\mu^{\X}(dx)\text{-a.s.}
\end{align*}
Now, we show that they are also equal $\tilde\Gamma$-a.s.~for almost every realization of $\tilde\Gamma$. Suppose that on some measurable set:
$H \subset {\X}$, we have, 
\begin{align*}
\int_\G\calB_i h(x, y)\eta_i(x,dy)\tilde\zeta_i(x) \ne \int_\G\calB_i h(x, y)\zeta_i(x,y)\tilde\eta(x,dy).
\end{align*}
Then, $\tilde\mu^{\X}(H) = 0$ and as a result $(t-s)\tilde\mu^{\X}(H) = 0$, which by \eqref{kur_gam} is equivalent to saying $\EE[\tilde\Gamma(H\times[s,t])] = 0$. But $\tilde\Gamma(H\times[s,t])$ is a non-negative random variable, so $\tilde\Gamma(H\times[s,t]) = 0$, $\PP$-a.s. This implies,
\begin{align*}
&\sum_{i=1}^{d_2+d_3}\int_{{\X}\times[s,t]}\int_\G\calB_i h(x, y)\eta_i(x,dy)\tilde\zeta_i(x)\tilde\Gamma(dx,du)\\
=& \sum_{i=1}^{d_2+d_3}\int_{{\X}\times[s,t]}\int_\G\calB_i h(x, y)\zeta_i(x,y)\tilde\eta(x,dy)\tilde\Gamma(dx,du), \quad \PP\text{-a.s.} 
\end{align*}
The left- and right-hand sides of the above equation match those from \eqref{eq:multi_B}.

Then, by setting $\Gamma_i(dx,dy,ds) = \eta_i(x,dy)\hat\Gamma_i(dx,ds)$, we have for any $h\in\calC^2_b(\X)$, the process defined in \eqref{mar:cha:ctr},
    \begin{align*}
        M^{\Gamma, h}_t &:= h(X_t) - \int_0^t \calA h(X_s)ds - \sum_{i=1}^{d_2+d_3}\int_{\X\times\G\times[0,t]}\calB_ih(x,y)\Gamma_i(dx,dy,ds) \\&= h(X_t) - \int_0^t \calA h(X_s)ds - \sum_{i=1}^{d_2+d_3}\int_{\X\times\G\times[0,t]}\calB_ih(x,y)\eta_i(x,dy)\hat\Gamma_i(dx,ds),
    \end{align*}
    is an $\calF_t$-martingale, which gives the admissibility.

\end{proof}

\subsection{Finiteness of the optimum}\label{sub33:min}
Another ingredient that we need for proving Theorem \ref{thm:main:min} is the finiteness of the optimum (values).

\begin{lemma}\label{ass:near:min}
Let Assumptions \ref{ass:21:min} and \ref{ass:22:min} be in force. The following bound holds: 
$$U_{(i)} = \inf_{\Gamma\in\calG_{(i)}}\Big\{\int_{\X} \ell(x)\nu^\Gamma(dx)+ \int_{\X\times \G}g(x)\cdot\mu^\Gamma(dx,dy)\Big\} < \infty.$$
\end{lemma}

\begin{proof}
This case deals with a minimization problem. Hence, it is sufficient to show that there exists at least one admissible control $\Gamma$, for which
    \begin{align*}
        \int_{\X} \ell(x)\nu^\Gamma(dx)+ \int_{\X\times \G}g(x)\cdot\mu^\Gamma(dx,dy) < \iy.
    \end{align*}
    Consider the trivial admissible control $\Gamma$ having $0$ mass, i.e., $\mu^\Gamma(\X\times\G) =0$, then we do not need to worry about the $g_i$ term for $i\in \lbrs{d_2+1,d_2+d_3}$. For the second term in \eqref{210a:min}, which is the original $f$ term, it follows by \eqref{g:cost:finite}. Finally, for the $\ell$ term, the argument follows by Assumptions $(A_3)$ and $(A_4)$; here we need the full power of $(A_4)$.
\end{proof}

\subsection{Sequentially compactness of the controls' associated measures}\label{sec:34:min}

In Proposition \ref{measure:compact:min}, we establish a compactness result, which essentially states that the family of measures associated with admissible controls is sequentially compact. It plays key ingredient in the proof of our main theorem for the single-agent problem, and will also be useful in Section \ref{sec4} when we study MFG. The proof of the proposition follows by Prokhorov's theorem, using tightness of measures, and requires a one-point compactification idea, inspired by the appendix of \cite{Budh-03}.

Before stating the proposition, we impose a further restriction on Case (i) without altering the problem's essence. We consider controls with bounded cost, defining the new set of admissible controls in Case (i) as follows: 
\begin{align}\label{cost:bound:2}
    \calG_{(i)}' = \{\Gamma\in\calG_{(i)}: J(\Gamma)\leq 2U_{(i)}\},
\end{align}
where recall that $U_{(i)}$ is finite by the last lemma. 
It is worth noting that one can opt for any constant $K > 1$ in place of $2$ here. The choice of $K$ remains arbitrary, as the control problem remains equivalent: our focus is on minimization, rendering the exclusion of controls with excessive costs inconsequential.

\begin{proposition}\label{measure:compact:min}
    Let Assumptions \ref{ass:21:min} and \ref{ass:22:min} be in force. Then the following sets are sequentially compact:
    \begin{align*}
        \{(\nu^\Gamma,\{\mu_i^\Gamma\}_{i\in\lbrs{d_2+d_3}})\}_{\Gamma\in \calG_{(i)}'}.
    \end{align*}
\end{proposition}
\begin{proof}[Proof of Proposition \ref{measure:compact:min}]
  
The proof in Case (i) will follow a slightly different technique. Briefly speaking, the unboundedness of $\ell$ makes the integral fail in convergence even if the measure converges in the weak sense. Hence, we use a one-point compactification method.

\vspace{5pt}
\noindent{{\it Step 1: One-point compactification.}}
Consider an arbitrary  sequence of measures $\{\nu^n,\{\mu^n_i\}_i)\}_n\subset \{(\nu^\Gamma,\{\mu^\Gamma_i\}_i)\}_{\Gamma\in\calG_{(i)}'}$. We will now compactify the spaces, and trivially extend the measures. Let us consider $\bar \X$, the one-point compactification of $\X$, and denote this point by $x_\iy$. Extend $\{\nu^n\}_n$ to $\{\bar\nu^n\}_n$ on $\bar \X$, in the following way: for $A \in \bar \X$, $\bar\nu^n(A) = \nu^n(A\cap\X)$. Similarly, let $\overline{{\X}\times\G}$ be the one-point compactifications of ${\X}\times\G$, denoting this point by $z_\infty$.
Extend $\{\mu^n_{i}\}_{n,i}$ to $\{\bar \mu^n_{i}\}_{n,i}$ on $\overline{{\X}\times\G}$,
in the following way: for $A \in \overline{{\X}\times\G}$, $\bar \mu^n_{i}(A) = \mu^n_{i}(A\cap({\X}\times\G))$.

Since the space $\bar \X$ is compact, the family $\{\bar\nu^n\}_n$ is tight. Therefore, by Prokhorov's theorem, by going to a subsequence if necessary, there exists a limiting measure, $\bar\nu\in\calP(\bar \X)$, such that $\bar\nu^n \to \bar\nu$ in the sense of weak convergence. Decompose $\bar\nu$ as $\bar\nu = (1-\rho)\nu +\rho \delta_{\{x_\iy\}}$, where $\nu \in \calP(\X)$. Similarly, the families $\{\bar \mu^n_{i}\}_{n}$ are tight. By going to a subsequence if necessary, there exist $\bar\mu_i \in\calM(\overline{\X \times\G})$ for each $i$, such that $\bar \mu^n_{i} \to \bar\mu_i$ in the sense of weak convergence. Decompose $\bar \mu_i$ as $\bar \mu_i = (1-\tilde\rho)\mu_i +\tilde\rho \mu_{i,\iy}$.

In order to conclude the sequential compactness we need to show that $\rho = \tilde \rho = 0$, which means $\bar \nu =\nu\in \calP(\X)$ and $\bar \mu_i=\mu_i\in\calM({\X \times\G})$, and that there exists a control $\Gamma'\in\calG_{(ii)}'$ such that $\nu=\nu^{\Gamma'}$ and $\mu=\mu^{\Gamma'}$.

\vspace{5pt}
\noindent{\it Step 2: Showing $\rho=\tilde \rho=0$ and $\nu=\nu^{\Gamma'}$ and $\mu=\mu^{\Gamma'}$ for some ${\Gamma'}\in\calG_{(i)}'$.} First, notice that by \eqref{cost:bound:2}, we have 
\begin{align*}
    \sup_{\Gamma\in\calG_{(i)}'} J(\Gamma) \leq 2U_{(i)},
\end{align*}
which implies
\begin{align}\label{case:ii:2beta}
    \limsup_{n\to\iy} J(\Gamma_n) \leq 2U_{(i)},
\end{align}
where $\Gamma_n$ are the controls associated with $\nu^n$ and $\mu^n$. Notice that by Assumption $(A_5)$, we have
\begin{align}
    &\liminf_{|x|\to\infty}\ell(x)>2U_{(i)} + \al,\label{ell:beta:eps}\\
    &\liminf_{|x|\to\infty}h_i(x)>2U_{(i)} + \al, \quad \forall i \in \lbrs{d_2+d_3},\label{h:beta:eps}
\end{align}
for any $\al > 0$.

Construct two sequences of continuous real-valued functions on $\overline{\X}$, $\{\ell^m\}_{m\ge1}$ and $\{g^m\}_{m\ge1}$, such that, $\ell^m \nearrow \ell$ on $\X$, $g^m \nearrow g$ on $\X\times\G$, and in addition 
\begin{align*}
  \ell^m(x_\infty) = 2U_{(i)}+\al\qquad\text{and}\qquad   g^m_i(x_\infty) = (2U_{(i)}+\al), \quad \forall i \in \lbrs{d_2+d_3}, \quad \forall m.
\end{align*}
 Such construction of $\ell^m$ and $g^m$, which are dominated by $\ell$ and $g$ yet have their respective limit, is possible by \eqref{ell:beta:eps} and \eqref{h:beta:eps}. For simplicity, let us denote
\begin{align*}
    g_\iy = (g^1_1(x_\iy),\dots,g^1_{d_2+d_3}(x_\iy)).
\end{align*}
Since $\bar\nu^n=\nu^n$ on $\X$ and $\bar\mu^n_i=\mu^n_i$ on $\X\times\mathbb G$, we have for all $n$,
\begin{align}\notag
    \int_{\overline{\X}} \ell^m(x)\bar\nu^n(dx)&\leq\int_{\X} \ell(x)\nu^n(dx),\\
    \notag
    \int_{\overline{\X\times\G}}g^m(x)\cdot\bar\mu^n(dx,dy)&\leq\int_{\overline{\X\times\G}}g(x)\cdot\mu^n(dx,dy).
\end{align}
As a result, 
\begin{align*}
    &(1-\rho)\int_{\X} \ell^m(x)\nu(dx) + \rho(2U_{(i)}+\al) +(1-\tilde \rho)\int_{\X\times \G}g^m(x)\cdot\mu(dx,dy)+\tilde \rho g_\iy\cdot\mu_\iy(z_\iy)
    \\&\quad= \int_{\overline{\X}} \ell^m(x) \bar\nu(dx) + \int_{\overline{{\X}\times\G}}g^m(x)\cdot\bar\mu(dx,dy)\\&\quad=\lim_{n\to\infty}\Big[\int_{\overline{\X}} \ell^m(x)\bar\nu^n(dx)+ \int_{\overline{{\X}\times\G}}g^m(x)\cdot\bar\mu^n(dx,dy)\Big]\\&\quad\leq \limsup_{n\to\infty}\Big[\int_{\X} \ell(x)\nu^n(dx)+ \int_{{\X}\times\G}g(x)\cdot\mu^n(dx,dy)\Big] \leq 2U_{(i)},
\end{align*}
Taking $m \to\infty$, we get,

\begin{align}\label{2rho:leq:2beta}
    (1-\rho)\int_{\X} \ell(x)\nu(dx) + \rho(2U_{(i)}+\al) +(1-\tilde \rho)\int_{\X\times \G}g(x)\cdot\mu(dx,dy)+\tilde \rho g_\iy\cdot\mu_\iy(z_\iy) \leq 2U_{(i)}.
\end{align}
Notice that the above inequality holds true independent of $\al > 0$ and also, since $\ell$ and $g$ are bounded from below, we have $\int_{\X} \ell(x)\nu(dx) >-\infty$ and $\int_{\X\times \G}g(x)\cdot\mu(dx,dy)>-\infty$. Therefore, we must have $\rho = \tilde \rho = 0$, otherwise we can choose $\al$ large enough so the inequality is violated. As a result, we also have,
\begin{align}\label{case:ii:leq:2beta}
    \int_{\X} \ell(x)\nu(dx)+\int_{\X\times \G}g(x)\cdot\mu(dx,dy)\leq2U_{(i)},
\end{align}
In conclusion, we have shown that $\rho=\tilde\rho=0$ and \eqref{case:ii:leq:2beta}. Now, recall that our test functions $h\in\calC^2_b(\X)$ have compact supports of first and second derivatives. Therefore, since $\rho=\tilde\rho=0$ we have:
\begin{align*}
  &\int_{{\X}}\calA h(x)\nu(dx)+\sum_{i=1}^{d_2+d_3}\int_{{\X}\times\G}(\calB_i h)(x,y)\mu_{i}(dx,dy)\\&\quad=\int_{\overline{\X}}\calA \bar h(x)\bar\nu(dx)+\sum_{i=1}^{d_2+d_3}\int_{\overline{\X\times\G}}(\calB_i \bar h)(x,y)\bar\mu_{i}(dx,dy)\\
  &\quad=0,  
\end{align*}
where $\bar h$ is an extension of $h$ such that $\bar h(x_\iy) = 0$.
Then, by Lemma \ref{lem:exi}, there exists a process $X$ that has the stationary distribution $\nu^{\Gamma'}=\nu$ and occupation measure $\{\mu^{\Gamma'}_i\}_i=\{\mu_i\}_i$. In other words, $\nu$ and $\mu$ are 
associated with some ${\Gamma'}\in\calG_{(i)}$. We are left to show that ${\Gamma'}\in\calG_{(i)}'$, which is given by \eqref{case:ii:leq:2beta}. Therefore, the relaxed control $\Gamma' \in \calG_{(i)}'$ is indeed admissible, and we are done.
\end{proof}

We are now ready to prove the main theorem for the single-agent control problem in Case (i). 
\subsection{Proof of Theorem \ref{thm:main:min}}\label{sec:35:min}
Notice that  instead of \eqref{case:ii:2beta}, we have here,
\begin{align*}
    \lim_{n\to\iy} J(\Gamma_n) = U_{(i)}.
\end{align*}
Then, we construct the functions $\ell^m$ and $g^m$ as given in the proof of  Proposition \ref{measure:compact:min}, with the small adaptation of setting the limit to be to $U_{(i)}+\al$. The rest of the proof follows the same lines as in the proof of  Proposition \ref{measure:compact:min} and instead of \eqref{case:ii:leq:2beta}, we have
\begin{align*}
    \int_{\X} \ell(x)\nu^\Gamma(dx)+\int_{\X\times \G}g(x)\cdot\mu^\Gamma(dx,dy) \leq U_{(i)}.
\end{align*}
But $U_{(i)}$ is the infimum, so we actually have
\begin{align*}
    \int_{\X} \ell(x)\nu^\Gamma(dx)+\int_{\X\times \G}g(x)\cdot\mu^\Gamma(dx,dy) = U_{(i)}.
\end{align*}
These show that $U_{(i)}$ is achieved under the control $\Gamma$.

\qed

\section{Proof of Theorem \ref{thm:main:max}}\label{sec3:max}
This section is devoted to the proof of Theorem \ref{thm:main:max}, which is the main result of the single-agent control problem in Case (ii). The proof is constructed in the same way as for Case (i) in Section \ref{sec3:min}. For completeness, we present all the necessary lemmas and omit the proofs when they are identical to those in the previous section.

\subsection{Characterization for admissible controls}\label{sub32:max}
As in Section \ref{sub32:min}, the next two lemmas characterize the relationship between admissible controls (excluding Condition (d) in Definition \ref{def21:max}) and stationary distribution using the generators $\calA$ and $\calB$, derived by a martingale problem. Notice that $d_2 = 0$ in Case (ii).
The proofs of the following two Lemmas are the same as the proofs for Lemmas \ref{lem:exo} and \ref{lem:exi} for Case (i).

\begin{lemma}\label{lem:exo:max}
Let $\Gamma$ be an admissible control with finite first moment, namely, $\EE[\Gamma(\X\times\G\times[0,1])]<\iy$ and for which the stationary distribution for state process $X^\Gamma$ is denoted by $\nu^\Gamma\in\calP({\X})$. Under Assumption 
$(B_2)$, there exist measures $ \{\mu_i^\Gamma\}_{i\in\lbrs{d_3}}\subset\calM({\X}\times\G)$, such that for all $h\in\calC^2_b({\X})$,  
\begin{align}\notag
\int_{{\X}}\calA h(x)\nu^\Gamma(dx)+\sum_{i=1}^{d_3}\int_{{\X}\times\G}(\calB_i h)(x,y)\mu_i^\Gamma(dx,dy)=0.
\end{align}
\end{lemma}

\begin{lemma}\label{lem:exi:max}
Let Assumption 
$(B_2)$ be in force. Let $\nu\in\calP({\X})$ and $ \{\mu_i\}_{i\in\lbrs{d_3}}\subset\calM({\X}\times\G)$, be such that for all $h\in\calC^2_b({\X})$,  
\begin{align}\label{eq:lem32:max}
\int_{{\X}}\calA h(x)\nu(dx)+\sum_{i=1}^{d_3}\int_{{\X}\times\G}(\calB_i h)(x,y)\mu_i(dx,dy)=0.
\end{align}
Let $\mu^{\X}_i$ be the state-marginal of $\mu_i$ and let $\eta_i$ be the transition function from ${\X}$ to $\G$ such that $\mu_i(dx,dy)=\eta_i(x,dy)\mu^{\X}_i(dx)$. Then, there exist a process $X$ and  random measures $\{\hat\Gamma_i\}_{i\in\lbrs{d_3}}$ on ${\X}\times\R_+$, adapted to $\calF_t$, such that:
\begin{itemize}
\item $X$ is stationary and for any $t\in\R_+$ $X(t)$ has the distribution $\nu$;
\item $\hat\Gamma_i$ has stationary increments, $\hat\Gamma_i({\X}\times[0,t])$ is finite for any $t\in[0,1]$, and $\E[\hat\Gamma_i(\cdot\times[0,t])]=t\mu_i^{\X}(\cdot)$;
\item for any $h\in\calC^2_b({\X})$, the process $(N^h_t)_{t\in[0,1]}$ is an $\calF_t$-martingale, where
\begin{align}\notag
N^h_t:= h(X_t)-\int_0^t\calA h(X_s)ds-\sum_{i=1}^{d_3}\int_{{\X}\times[0,t]}\int_\G\calB_i h(x, y)\eta_i(x,dy)\hat\Gamma_i(dx,ds).
\end{align}
\end{itemize}
In particular, set $\Gamma_i(dx,dy,ds) := \eta_i(x,dy)\hat\Gamma_i(dx,ds)$ for $i \in \lbrs{1,d_3}$. Then, if condition (d) from Definition \ref{def21:max} holds, $\Gamma$ is an admissible control. As a result, for any $C\subseteq\X\times\G$,
\begin{align}\label{eq:mu-gamma:max}
\mu_i(C)=\EE[\Gamma_i(C\times[0,1])].
\end{align}
\end{lemma}

\begin{remark}\label{rem:GammaToY}
    Notice we have not shown yet that Condition (d) from Definition \ref{def21:max} holds (which is only necessary for Case (ii)). However, when we use this lemma in the sequel (specifically, in the proof of Proposition \ref{measure:compact:max}, in Subsection \ref{sec:35:max}) we will establish the admissibility of this control $\Gamma$. In other occasions when we use this lemma, the admissibility conditions are already held by construction or assumption. Conditions (a)-(c) of Definition \ref{def21:max} are satisfied in any case, and in particular, an adapted $\Gamma$ can be constructed, see \cite[Remark 1.8, Lemma 6.1]{kur-sto}.
\end{remark}

\subsection{Finiteness of the optimum}\label{sub33:max}
We now show the finiteness of the optimum (values) in Case (ii).

\begin{lemma}\label{ass:near}
Let Assumptions \ref{ass:21:max}, \ref{k_matrix}, and \ref{ass:22:max} be in force. The following bound holds: 
$$U_{(ii)} = \sup_{\Gamma\in\calG_{(ii)}}\Big\{\int_{\X} \ell(x)\nu^\Gamma(dx)+ \int_{\X\times \G}g(x)\cdot\mu^\Gamma(dx,dy)\Big\} < \infty.$$
\end{lemma}

\begin{proof}
First, notice that
    \begin{align*}
        &\sup_{\Gamma\in\calG_{(ii)}}\Big\{\int_{\X} \ell(x)\nu^\Gamma(dx)+  \int_{\X\times \G}g(x)\cdot\mu^\Gamma(dx,dy)\Big\} \\
        &\quad\leq \sup_{\Gamma\in\calG_{(ii)}}\int_{\X} \ell(x)\nu^\Gamma(dx)+ \sup_{\Gamma\in\calG_{(ii)}}\int_{\X\times \G}g(x)\cdot\mu^\Gamma(dx,dy).
    \end{align*}
    Therefore, it is sufficient to show that each of the terms is finite. The first term's finiteness is trivial, given that $\ell$ is bounded on $X$ by $(B_6)$. 
    For the second term, recall that $g$ is bounded by $(B_6)$ as well. Hence, we have
    \begin{align*}
        \sup_{\Gamma\in\calG_{(ii)}}\int_{\X\times \G}g(x)\cdot\mu^\Gamma(dx,dy) \leq \sup_{x\in \X} g(x) \cdot \sup_{\Gamma\in\calG_{(ii)}} \mu^\Gamma({\X}\times\G).
    \end{align*}
    Therefore it is sufficient to show that $\sup_{\Gamma\in\calG_{(ii)}}\mu^\Gamma({\X}\times\G)< \iy$. 
    Take an arbitrary $\Gamma$. Note we have 
    \begin{align}\label{eq:Y_mu}
        \mu_i^\Gamma(\X\times\G)=\EE\big[\Gamma_i(\X\times\G\times[0,1])\big].
    \end{align}
Define the stopping times $\tau_n:= \inf\{t\geq 0, |X_t^\Gamma|\geq n\}\wedge 1$. Notice that $|X_{\tau_n}^\Gamma|\leq n$, because the contribution due to the singular control is only downwards (i.e., reducing the norm) and we assume in Case (ii) that $\phi\equiv 0$. Define functions $h_n\in\calC^2_b(\X)$, such that for any $|x|\leq 2n$, we have $h_n(x) = x_1$.
This is the projection onto the first coordinate inside the ball $\{x:|x|\le 2n\}$. Therefore, $(M_t^{\Gamma, h_{n}})_{t\in[0,1]}$ is a martingale. Since $\tau_n$ is almost surely bounded by $1$, in particular, we have by the optional sampling theorem that
\begin{align*}
    \EE\big[M_{\tau_n}^{\Gamma, h_{n}}\big] = \EE\big[M_0^{\Gamma, h_{n}}\big].
\end{align*}
Now, let us calculate $M_{\tau_n}^{\Gamma, h_{n}}$. Notice that for $s \in [0,\tau_n]$, we have $|X_s^\Gamma| \leq n$, and therefore the function $h_n$ applies as the projection onto the first coordinate. We have
\begin{align*}
    M_{\tau_n}^{\Gamma, h_{n}} = X_{1,\tau_n}-\int_0^{\tau_n} \beta_1(X_s) ds -\sum_{i=1}^{d_3}\int_{\X\times\G\times[0,\tau_n]}\calB_ih_n(x,y)\Gamma_i(dx,dy,ds).
\end{align*}
Recall that in Case (ii), $d_2 = 0$; hence, the sum's upper limit is $d_2+d_3=d_3$. For $y_i > 0$, we have
\begin{align*}
    \calB_ih_n(x,y) = \frac{x_1-\gamma_1(x)y-x_1}{y_i} = \frac{-\gamma_1(x)y}{y_i},
\end{align*}
where $\gamma_1(x)$ is the first row of $\gamma(x)$. Notice by Assumption \ref{k_matrix}, and since $y \geq 0$ component-wise, we have
\begin{align*}
    {\bar\gamma} \leq \gamma_{1,i}(x) = \frac{\gamma_{1,i}(x)y_i}{y_i}\leq \frac{\gamma_1(x)y}{y_i} = -\calB_ih_n(x,y).
\end{align*}
On the other hand, for $y_i = 0$, we still have
\begin{align*}
    {\bar\gamma} \leq \gamma_{i,1}^T(x) = \gamma_i^T(x)\cdot \nabla h_n(x) = -\calB_ih_n(x,y).
\end{align*}
As a result,
\begin{align*}
d_3{\bar\gamma}\sum_{i=1}^{d_3}\Gamma_i(\X\times\G\times[0,\tau_n]) 
&\leq -\sum_{i=1}^{d_3}\int_{\X\times\G\times[0,\tau_n]}\calB_ih_n(x,y)\Gamma_i(dx,dy,ds) \\
&= M_{\tau_n}^{\Gamma, h_{n}} - X_{1,\tau_n}+\int_0^{\tau_n} \beta_1(X_s) ds.
\end{align*}
This in turn implies
\begin{align*}    d_3{\bar\gamma}\sum_{i=1}^{d_3}\EE\big[\Gamma_i(\X\times\G\times[0,\tau_n])\big] &\leq \EE\Big[M_{\tau_n}^{\Gamma, h_{n}} - X_{1,\tau_n}+\int_0^{\tau_n} \beta_1(X_s) ds\Big] \\&\leq\EE\Big[M_{\tau_n}^{\Gamma, h_{n}} +\int_0^{\tau_n} \beta_1(X_s) ds\Big] \\&\leq \EE\big[\big|M_0^{\Gamma, h_{n}}\big|\big]  + \EE\Big[\int_0^{\tau_n} |\beta_1(X_s)| ds\Big]\\&\leq \EE[|X_{1,0}|] + \EE\Big[\int_0^1 |\beta_1(X_s)| ds\Big] \\&=\int_{\X}  (|x_1|+|\beta_1(x)|)\nu^\Gamma(dx),
\end{align*}
where the second inequality is because $X_{1,\tau_n}\geq 0$ since for any $t\in[0,1]$, $X_t\in \R^d_+$; the fourth inequality follows by direct definition of $M_0^{\Gamma, h_{n}}$ and since $\tau_n\le 1$; the last equality is by stationarity of $X$. Now take $n \to \iy$, and by the monotone convergence theorem we have
\begin{align*}
    d_3{\bar\gamma}\sum_{i=1}^{d_3}\EE\big[\Gamma_i(\X\times\G\times[0,1])\big] \leq \int_{\X} (|x_1|+|\beta_1(x)|)\nu^\Gamma(dx).
\end{align*}
Recall by Assumption $(B_2)$ and the definition of $\bar V$ given in Remark \ref{remark21} (with $K_r=4r\ge 4$), we have that $|x_1| +|\beta_1(x)| \sim o(\Vfour)$, we get by \eqref{Vfour} that,
    \begin{align*}
    \sup_{\Gamma\in\calG_{(ii)}}\int_{\X}  (|x_1|+|\beta_1(x)|)\nu^\Gamma(dx) \leq \sup_{\Gamma\in\calG_{(ii)}}\int_{\X}\Vfour(x) \nu^\Gamma(dx)\leq \frac{k_1}{k_2}< \iy.
    \end{align*}
This gives
\begin{align}\label{Y:finite}
    \sup_{\Gamma\in\calG_{(ii)}}\sum_{i=1}^{d_3}\mu_i^\Gamma(\X\times\G) < \iy,
\end{align}
which concludes the proof.
\end{proof}

\subsection{Sequentially compactness of the controls' associated measures}\label{sec:34:max}

We now establish the sequentially compactness result in Case (ii). The proof of the proposition follows by Prokhorov's theorem, using tightness of measures.

\begin{proposition}\label{measure:compact:max}
    Let Assumptions \ref{ass:21:max}, \ref{k_matrix}, and \ref{ass:22:max} be in force. Then the following sets are sequentially compact:
    \begin{align*}
        \{(\nu^\Gamma,\{\mu_i^\Gamma\}_{i\in \lbrs{d_3}})\}_{\Gamma\in \calG_{(ii)}}.
    \end{align*}
\end{proposition}
\begin{proof}[Proof of Proposition \ref{measure:compact:max}]
  
We show the tightness of measures and then show that the limiting measures (along a converging subsequence) are associated with an admissible control. Recall that in Case (ii), $d_2 = 0$; hence, in this case $i\in\lbrs{d_2+d_3}=\lbrs{d_3}$.

\vspace{5pt}
\noindent{{\it Step 1: Tightness.} Notice that by Remark \ref{rem:23}, we have
\begin{align*}
\sup_{\Gamma\in\calG_{(ii)}}\EE[\Vfour(X^\Gamma_0)]\leq \frac{k_1}{k_2}<\iy.
\end{align*}
By the construction of $\Vfour$, we have $|x|\sim o(\Vfour)$, and therefore
\begin{align*}
\sup_{\Gamma\in\calG_{(ii)}}\EE[|X^\Gamma_0|]<\iy.
\end{align*}
Next, set
\begin{align*}
    f(R):=\sup_{\Gamma\in \calG_{(ii)}}\frac{\EE[|X^\Gamma_0|]}{R}.
\end{align*}
So, $f(R)\to0$ as $R\to\iy$.
In particular, this implies by Markov inequality,
\begin{align*}
    \nu^\Gamma(\{x: |x|>R\}) = \PP\big(|X^{\Gamma}_0|>R\big) \leq f(R).
\end{align*}
This implies the tightness of $\{\nu^\Gamma\}_{\Gamma\in\calG_{(ii)}}$. 

The proof that $\{\{\mu^\Gamma_i\}_{i\in\lbrs{d_3}}\}_{\Gamma\in\calG_{(ii)}}$ is tight is more demanding as the measures $\mu_i$ are over $\X\times \G$, and therefore, the uniform bound on the expectation of $|X|$ cannot be directly used.

Fix an arbitrary $i\in\lbrs{d_3}$. It is sufficient to find a function $q:\R_+\to \R$, such that $\lim_{R\to\iy}q(R) = 0$, and for all $n$,
\begin{align}\label{eq:q}
    \mu_i^n(C_R) \leq q(R),
\end{align}
where $\X\times\G\supseteq C_R := \{(x,y):|x|+|y| > R\}$. Recall the definition of $\mu_i^n$ in Lemma \ref{lem:exo:max} and \eqref{Gamma}. Then, 
\begin{align*}
    \mu_i^n(C_R)&=\EE\big[\Gamma_i^n(C_R\times[0,1])\big].
\end{align*}
As in the proof of Lemma \ref{ass:near}, define the stopping times $\tau_n:= \inf\{t\geq 0, |X_t^\Gamma|\geq n\}\wedge 1$. Also, take functions $h_{n,j}\in\calC^2_b(\X)$, such that for $|x|\leq 2n$, we have $h_{n,j}(x) = x_j^2$. This is the square of the projection onto the $j$-th coordinate inside the ball $\{x:|x|\le 2n\}$. For $y_i>0$, we have 
\begin{align*}
    \calB_ih_{n,j}(x,y) = \frac{(x_j-\gamma_j(x)y)^2-x_j^2}{y_i} = \frac{-2x_j\gamma_j(x)y+(\gamma_j(x)y)^2}{y_i}.
\end{align*}
Then we have
\begin{align}\label{Gxy:ine}
    \frac 12 {\bar\gamma} x_j +\frac 12 {\bar\gamma}^2|y|_1\leq \frac 12(\gamma_{j,i}(x)x_j + \gamma_{j,i}(x)\gamma_{j}(x)y)\leq \gamma_{j,i}(x)x_j,
\end{align}
where $|\cdot|$ is the $l_1$-norm, the first inequality is because $\gamma(x)$ is bounded below by ${\bar\gamma}$, and we have the second inequality without affecting the calculation, because $X\geq 0$, and therefore $x_j \geq \gamma_{j}(x)y\geq 0$; otherwise $\Gamma$ has no mass.
Further, we have
\begin{align*}
    \gamma_{j,i}(x)x_j\leq \frac{x_j\gamma_{j}(x)y}{y_i} \leq \frac{x_j\gamma_j(x)y+(x_j\gamma_{j}(x)y-(\gamma_j(x)y)^2)}{y_i} = -\calB_ih_{n,j}(x,y),
\end{align*}
where again we use $x_j \geq \gamma_{j}(x)y\geq 0$. For $y_i = 0$, we have
\begin{align*}
    \calB_ih_{n,j}(x,y) = \gamma_i^T(x)\cdot\nabla h_{n,j}(x) = -2\gamma_{j,i}(x)x_j.
\end{align*}
Then by \eqref{Gxy:ine} we have
\begin{align*}
    \frac 12 {\bar\gamma} x_j +\frac 12 {\bar\gamma}^2|y|_1\leq -\calB_ih_{n,j}(x,y).
\end{align*}
Then, we have
\begin{align*}
&\sum_{i=1}^{d_3}\EE\Big[\frac12\int_{\X\times\G\times[0,\tau_n]}\big({\bar\gamma} x_j+{\bar\gamma}^2|y|_1\big)\Gamma_i(dx,dy,ds)\Big]\\ 
&\qquad\leq \EE\Big[-\sum_{i=1}^{d_3}\int_{\X\times\G\times[0,\tau_n]}\calB_ih_{n,j}(x,y)\Gamma_i(dx,dy,ds)\Big] \\
&\qquad= \EE\Big[M_{\tau_n}^{\Gamma, h_{n}} - X_{j,\tau_n}^2+\int_0^{\tau_n} \big(2X_{j,s}\beta_j(X_s)+a_{j,j}(X_s) \big)ds\Big]\\
&\qquad\leq\EE\Big[M_{\tau_n}^{\Gamma, h_{n}} +\int_0^{\tau_n} \big(2X_{j,s}\beta_j(X_s)+a_{j,j}(X_s)\big)ds\Big] \\
&\qquad\leq \EE\big[\big|M_0^{\Gamma, h_{n}}\big|\big]  + \EE\Big[\int_0^{\tau_n} |2X_{j,s}\beta_j(X_s)+a_{j,j}(X_s)| ds\Big]\\
&\qquad\leq \EE[|X_{j,0}^2|] + \EE\Big[\int_0^1 |2X_{j,s}\beta_j(X_s)+a_{j,j}(X_s)| ds\Big] \\
&\qquad=\int_{\X} \big(|x_j|^2+|2x_j\beta_j(x)+a_{j,j}(x)|\big)\nu^\Gamma(dx).
\end{align*}
Now take $n \to \iy$, and use the monotone convergence theorem to get
\begin{align*}
    \sum_{i=1}^{d_3}\EE\Big[\frac12\int_{\X\times\G\times[0,1]}\big({\bar\gamma} x_j+{\bar\gamma}^2|y|_1\big)\Gamma_i(dx,dy,ds)\Big] \leq \int_{\X} \big(|x_j|^2+|2x_j\beta_j(x)+a_{j,j}(x)|\big)\nu^\Gamma(dx).
\end{align*}
Sum up over $j$:
\begin{align*}
    \sum_{i=1}^{d_3}\EE\Big[\frac12\int_{\X\times\G\times[0,1]}\big({\bar\gamma}|x|_1+d{\bar\gamma}^2|y|_1\big)\Gamma_i(dx,dy,ds)\Big] \leq \int_{\X} \sum_{j=1}^{d}\big(|x_j|^2+|2x_j\beta_j(x)+a_{j,j}(x)|\big)\nu^\Gamma(dx),
\end{align*}
where $d$ in $d{\bar\gamma}^2$ is the dimension of $\X$.
Note that $\sum_{j=1}^{d}\big(|x_j|^2+|2x_j\beta_j(x)+a_{j,j}(x)|\big)\sim o(\Vfour)$. By the equivalence between the $l_1$ and $l_2$ norms, we get
\begin{align*}
    \sup_{\Gamma \in \calG_{(ii)}} \sum_{i=1}^{d_3}\int_{\X\times \G} \big(|x|+|y|\big)\mu_i^\Gamma(dx,dy) < \iy.
\end{align*}
Using generalized Markov's inequality for measures we get
\begin{align*}
    \sup_{\Gamma \in \calG_{(ii)}}\mu_i(\{|x|+|y|\geq R\}) \leq \frac{\sup_{\Gamma \in \calG_{(ii)}} \sum_{i=1}^{d_3}\int_{\X\times \G} \big(|x|+|y|\big)\mu_i^\Gamma(dx,dy)}{R}.
\end{align*}
Set the right-hand side to be $q(R)$, which indeed $\to 0$ as $R\to\iy$ establishes the tightness of $\{\mu^\Gamma_i\}_i$.

\vspace{5pt}

\noindent{{\it Step 2: Limiting control and its admissibility}.} Consider an arbitrary  sequence of measures $\{\nu^n,\{\mu^n_i\}_i)\}_n\subset \{(\nu^\Gamma,\{\mu^\Gamma_i\}_i)\}_{\Gamma\in\calG_{(ii)}}$. By the tightness established earlier, we may go along a subsequence, which we again label by $\{n\}$, and obtain a limit point $(\nu^n,\{\mu^n_i\}_i)\to (\nu,\{\mu_i\}_i)$
in the weak convergence sense. Hence,
\begin{align*}
  &\int_{{\X}}\calA h(x)\nu(dx)+\sum_{i=1}^{d_3}\int_{{\X}\times\G}(\calB_i h)(x,y)\mu_{i}(dx,dy)\\
  &\quad=\lim_{n\to\iy}\int_{\X}\calA h(x)\nu^n(dx)+\sum_{i=1}^{d_3}\int_{{\X\times\G}}(\calB_i h)(x,y)\mu_{i}^n(dx,dy)\\
  &\quad=0. 
\end{align*}
Next, by Lemma \ref{lem:exi:max}, 
there exist random measures $\Gamma=\{\Gamma_i\}_i$, satisfying Conditions (a)-(c)in Definition \ref{def21:max} and such that $(\nu,\{\mu_i\}_i)= (\nu^\Gamma,\{\mu^\Gamma_i\}_i)$. 
We are left to show that Condition (d) in Definition \ref{def21:max} holds.

Consider the controls $\Gamma^n$ and the measures $\nu^n$ defined earlier in the proof. Since $\Gamma^n \in \calG_{(ii)}$, Condition (d) holds for $\nu^n$, we have
\begin{align*}
    0\leq \int_\X\calA \Vfour(x) \nu^n(dx).
\end{align*}
Recall Remark \ref{remark21}, we have
\begin{align*}
    \calA \Vfour(x)-k_1\leq -k_2\Vfour(x) \leq 0,
\end{align*}
or
\begin{align*}
    -\calA \Vfour(x) \geq -k_1.
\end{align*}
This lower bound for $-\calA \Vfour$ enables us to apply the monotone convergence theorem when we construct a sequence $A_m$ below. 
Now, set a sequence of functions $A_m$ approximating $-\calA \Vfour(x)$, where for each $m$,
\begin{itemize}
    \item $A_m(x) = -\calA \Vfour(x)$ for $|x| \leq R_m$;
    \item $A_m(x)$ is continuous and bounded, and the lower bound is uniform for all $m$;
    \item $A_m(x)\nearrow -\calA \Vfour(x)$, pointwise, as $m \to \iy$;
\end{itemize}
where we take an arbitrary sequence $R_m \nearrow \iy$. This limit is necessary, otherwise, we cannot have  $A_m(x)\nearrow -\calA \Vfour(x)$ since the latter is not bounded while the $A_m$'s are. By construction, we have for all $m, n$,
\begin{align*}
    \int_{\X} A_m(x) \nu^n(dx) \leq \int_{\X} -\calA \Vfour(x)\nu^n(dx) \leq 0. 
\end{align*}
Further, since each $A_m$ is continuous and bounded, by weak convergence
\begin{align*}
     \int_{\X} A_m(x) \nu^n(dx) \to \int_{\X} A_m(x) \nu(dx), \quad \text{as} \quad n \to \iy,
\end{align*}
for any fixed $m$. Therefore,
\begin{align*}
    \int_{\X} A_m(x) \nu(dx) = \lim_{n\to\iy}\int_{\X} A_m(x) \nu^n(dx)\leq 0,
\end{align*}
for any fixed $m$. Apply the monotone convergence theorem, we have,
\begin{align*}
    \int_{\X} A_m(x) \nu(dx) \nearrow  \int_{\X} -\calA \Vfour(x)\nu(dx),\quad \text{as} \quad m \to \iy.
\end{align*}
As a result,
\begin{align*}
    \int_{\X} -\calA \Vfour(x)\nu(dx) = \lim_{m\to\iy}\int_{\X} A_m(x) \nu(dx)\leq 0,
\end{align*}
which implies Condition (d) in Definition \ref{def21:max}.
}\end{proof}

We are now ready to prove the main theorem for the single-agent control problem. 
\subsection{Proof of Theorem \ref{thm:main:max}}\label{sec:35:max}
The finiteness of the value follows from Lemma \ref{ass:near}. 
Let $\{\Gamma_n\}_n$ be a sequence of admissible controls such that $\int_{\X} \ell(x)\nu^{\Gamma^n}(dx)+\int_{\X\times \G}g(x)\cdot\mu^{\Gamma^n}(dx,dy)$ converges to $U_{(ii)}$, according to the relevant case, as $n \to \iy$. By Proposition \ref{measure:compact:max}, by going to a subsequence if necessary, there exist measures $\nu$ and $\mu$, associated with some admissible control $\Gamma$, and $\nu^{\Gamma^n}\to\nu^\Gamma$ and $\mu^{\Gamma^n}\to\mu^\Gamma$ in the sense of weak convergence. 

Since both $\ell$ and $g$ are continuous and bounded by Assumption \ref{ass:22:max} $(B_5), (B_6)$, and weak convergence, we get
\begin{align*}
    &\int_{\X} \ell(x)\nu^\Gamma(dx) + \int_{\X\times \G}g(x)\cdot\mu^\Gamma(dx,dy) \\&=\lim_{n\to\iy}\int_{\X} \ell(x)\nu^{\Gamma^n}(dx) + \int_{\X\times \G}g(x)\cdot\mu^{\Gamma^n}(dx,dy)\\&= U_{(ii)}.
\end{align*}

\qed

\section{The mean field game}\label{sec4}


We now consider the MFG in the stationary multidimensional setup with singular controls. 
\red{The results of single-agent control problems have several applications in the context of MFGs. For instance, in biological harvesting models, the reward-maximization framework can be adapted to MFGs to model sustainable population dynamics, where multiple agents (representing different harvesting strategies) interact to optimize resource extraction while maintaining ecological balance. Similarly, in queueing systems under heavy traffic, the cost-minimizing framework for singular controls with reflection boundaries is directly applicable to MFGs modeling congestion and delays, where agents (e.g., customers or servers) optimize their strategies to minimize system-wide costs associated with waiting times and idle periods. Potentially, the theory can also be extended to financial and investment models, where investors optimize their portfolios or consumption strategies, considering the market's overall dynamics and accounting for transaction costs or market frictions. Supply chain and manufacturing systems also benefit from this framework, with singular controls modeling production shutdowns or inventory constraints, while environmental resource management models can address collective efforts like pollution reduction or sustainable resource usage, optimizing individual contributions to a common goal. In all these cases, the MFG framework allows for the study of equilibrium strategies in complex systems where individual actions influence both local and global outcomes, providing insights into optimal behavior under various constraints and interactions.}

\red{We now turn to the analysis.} The existence of an optimal control for the single-agent problem, \red{and the sequential compactness of measures in Propositions \ref{measure:compact:min} and \ref{measure:compact:max}} are pivotal in the current analysis. However, given these results, the analysis in this section is the same for both the (harvesting-inspired) maximization and (queueing-inspired) minimization problems, except for a minor point that we point out in Footnote \ref{footnote} below. Hence, we only present the maximization setup, namely, Case (ii).

The MFG equilibrium is defined as the fixed point of a best response mapping. We consider the same type of dynamics and keep the same definition for admissible controls. We use the notation $\calG$ for the set of all admissible controls. Saying this, we allow the reward to depend on a probability measure in addition to the underlying state. 
This probability measure captures the behavior of the mean field term (which in turn mimics the limiting behavior in large population games). We formulate the reward using the reward structure from \eqref{210a:max}. Specifically, for an admissible control $\Gamma\in\calG_{(ii)}$, and a probability measure $\tilde\nu\in\calP(\X)$, consider the {\it reward function in linear programming form} 
\begin{align}\label{cost:MFG}
J(\Gamma,\tilde\nu):=\int_{\X} \ell(x,\tilde\nu)\nu^\Gamma(dx)+ \int_{\X\times \G}g(x,\tilde \nu)\cdot\mu^\Gamma(dx,dy),
\end{align}
where $\nu^\Gamma$ and $\mu^\Gamma$ are the measures associated with the control $\Gamma$ as described in  Lemma \ref{lem:exo:max}. 
Hence, the measures denoted as $\mu$ in the above and following contexts actually represent $\mu$ and $\lambda$. Such notation is consistent with our single-agent result. The measure $\tilde \nu$ inside the reward function is the mean field term. 
Fixing $\tilde \nu\in\calP(\X)$, we find the optimal control $\Gamma$ under this criterion \eqref{cost:MFG}. Then, we are looking for the fixed point in the sense that $\tilde \nu = \nu^{\Gamma}$. 
We will show the existence of such a fixed point by applying Kakutani--Glicksberg--Fan fixed point theorem. 
Assumptions \ref{ass:21:max} and \ref{k_matrix} are in force throughout this section, and we extend Assumption \ref{ass:22:max} to the MFG framework in a minimal manner by imposing the same conditions as in the single-agent model, along with Lipschitz continuity in the measure-component, a standard requirement in the literature.
\begin{assumption}\label{ass:MFG}
\begin{enumerate}
    \item[($C_5$)] 
    For all $\nu$, the mapping $\ell(\cdot,\nu): \X \mapsto \R$, and $g(\cdot,\nu): \X \mapsto \R_+$ satisfy Assumption \ref{ass:22:max}.
\item[($C_6$)]
Uniformly in all $x \in \X$, the mapping $\ell(x,\cdot), \calP(\X) \mapsto \R$, and $g(x,\cdot), \calP(\X) \mapsto \R_+$ are globally Lipschitz in Prokhorov distance, with possibly different Lipschitz constants.
\end{enumerate}
\end{assumption}

In \cite[Assumption 5]{Roxana2021}, the authors provide a specific form of the running reward dependence on the measure, namely, they take into account the integral of an auxiliary continuous and bounded function against the measures. We do not specify the form, but we do require in $(C_6)$ that it is Lipschitz. Notice that in our state dynamic, the coefficients do not depend on the measures, this is similar to their Assumption $5(5)(a)$. They also require in Assumption $5(4)$ that the initial distribution is square integrable, and in our case, it is taken care of by the Lyapunov function.

Define the sets\footnote{\red{In the definition of $\mathbb A'$ we require $\int_\X\calA\Vfour(x)\nu(dx)\geq 0$. For the minimization problem (Case (i), this requirement is not necessary. In addition, for the minimization problem, the indices for $\mu$ should be changed to $i\in\lbrs{d_2+d_3}$ since we are merging  $\Phi$ and $\Gamma$ as described in Remark \ref{rem:mu:la}.\label{footnote}}} 
\begin{align*}
    \mathbb A &:= \left\{\nu\in\calP({\X}): \exists \Gamma \in \calG;\forall t\in[0,1], \calL(X^\Gamma_t) = \nu, \right\}, 
\\
\mathbb A’ &:= \Big\{\nu\in\calP({\X}): \exists\{\mu_i\}_{i\in\lbrs{d_3}}\subset\calM({\X}\times\G),\;\forall h\in \calC_b^2({\X}),\\
&\hspace{3cm}\int_{{\X}}\calA h(x)\nu(dx)+\sum_{i=1}^{d_3}\int_{{\X}\times\G}(\calB_i h)(x,y)\mu_i(dx,dy)=0, \int_\X\calA\Vfour(x)\nu(dx)\geq 0 \Big\}.
\end{align*}
These sets are not empty as they contain the stationary distribution implied by the trivial control $\Gamma(\X\times\G\times[0,1])= 0$.
From Lemma \ref{lem:exo:max}, we have $\mathbb A \subseteq \mathbb A’$, and from Lemma \ref{lem:exi:max}, we get that $\mathbb A’ \subseteq \mathbb A$. Therefore, $\mathbb A = \mathbb A’$, and we will use both sets interchangeably.

For a measure $\tilde \nu\in\calP({\X})$, we set:
\begin{align*}
    \calG(\tilde\nu) := \sup_{\nu\in\mathbb A,\;\mu \in \calI(\nu)}\Big\{\int_{\X} \ell(x,\tilde\nu)\nu(dx)+ \int_{\X\times \G}g(x,\tilde \nu)\cdot\mu(dx,dy)\Big\},
\end{align*}
where for $\nu \in \mathbb A$, $\calI(\nu) \subseteq (\calM({\X}\times\G))^{d_3}$ is given by
\begin{align*}
    \calI(\nu) &:= \Big\{\{\mu_i\}_{i\in\lbrs{d_3}}\subset\calM({\X}\times\G): \forall h\in \calC_b^2({\X}),\\
    &\hspace{5cm}\int_{{\X}}\calA h(x)\nu(dx)+\sum_{i=1}^{d_3}\int_{{\X}\times\G}(\calB_i h)(x,y)\mu_i(dx,dy)=0  \Big\}.
\end{align*}
The set $\calI(\nu)$ represents the set of measures $\mu=\{\mu_i\}_{i\in\lbrs{d_3}}$ for which $(\nu,\mu)$ satisfy the condition of Lemma \ref{lem:exi:max}, and $\calK(\tilde \nu)$ represents the optimum one may achieve assuming the mean field term is $\tilde \nu$. 
\begin{definition}
A probability measure $\hat\nu \in \mathbb A$ is called an {\rm MFG} equilibrium if there exists $\{\hat\mu_i\}_{i\in\lbrs{d_3}}\in\calI(\hat\nu)$, such that, 
\begin{align*}
    \int_{\X} \ell(x,\hat\nu)\hat\nu(dx)+ \int_{\X\times \G}g(x,\hat \nu)\cdot\hat\mu(dx,dy)  = \calK(\hat\nu).
\end{align*}

\end{definition}
\begin{theorem}\label{exi:LP}
    Under Assumptions \ref{ass:21:max}, and \ref{ass:MFG}, there exists an MFG equilibrium.
\end{theorem}

The existence of an MFG equilibrium uses Kakutani--Glicksberg--Fan fixed point theorem. For this, we define a proper set-valued mapping satisfying several continuity, convexity, and compactness properties. In what follows, we provide the necessary building blocks for the fixed point argument.

Define the set 
\begin{align*}
\mathbb V := \{(\nu,\mu):\nu\in\mathbb A, \mu \in \calI(\nu)\}.
\end{align*}
On the set $\mathbb V$, we take the product metric, i.e.,
\begin{align}\label{MFG:Prok}
    d_{\mathbb V}((\nu_1,\mu_1),(\nu_2,\mu_2)) := \sup\{d_P(\nu_1,\nu_2),d_P(\mu_1,\mu_2)\},
\end{align}
where $d_P$ is the Prokhorov metric.

While the Prokhorov's distance between two probability measures is well-used, we need to formulate the proper notion of this metric for $\calM({\X}\times\G)$. Note that the space $\X \times \G$ is separable metric spaces. Hence, by \cite{MR94838}, they are metrizable. In particular, such a metric corresponds to the weak topology, i.e., $\mu_n \to \mu$ if and only if
\begin{align*}
    \int hd\mu_n \to \int hd\mu,
\end{align*}
for all bounded continuous functions $h$. We denote this specific metric by $d_P(\cdot,\cdot)$, and refer to it as the {\it generalized Prokhorov metric}.

\begin{lemma}
    The set $\mathbb V$ is compact and convex.
\end{lemma}
\begin{proof}

    Note that $\mathbb V$ is metrizable. Hence, it is sufficient to consider sequential compactness. Take an arbitrary sequence $\{(\nu^n,\mu^n)\}_n\subset\mathbb V$. By Proposition \ref{measure:compact:max}, by going to a subsequence $\nu^{n_k}$ if necessary, there exists $\nu^\Gamma$ associated with some admissible $\Gamma$ such that $\nu^{n_k} \to \nu^\Gamma$ in the sense of weak convergence. Now consider the sequence $\mu^{n_k}$, again by Proposition \ref{measure:compact:max}, by going to a subsequence $\mu^{n_{k_j}}$ if necessary, there exists $\mu^\Gamma$ associated with some admissible $\Gamma$ such that $\mu^{n_{k_j}} \to \mu^\Gamma$ in the sense of weak convergence. By the construction of metric \eqref{MFG:Prok}, $(\nu^{n_{k_j}},\mu^{n_{k_j}})$ is a converging subsequence of the original sequence, further, its limit point $(\nu^\Gamma,\mu^\Gamma)\in\mathbb V$ as the measures are associated with some control $\Gamma$.
    
    For convexity, for any $(\nu_1,\mu_1),(\nu_2,\mu_2) \in \mathbb V$ and any $\lambda\in[0,1]$, let $\nu’ = \lambda\nu_1+(1-\lambda)\nu_2$ and similarly for $\mu’$. Then it’s easy to verify that $\nu’ \in \mathbb A$ with $\mu’\in\calI(\nu’)$, so $(\nu’,\mu’)\in\mathbb V$.
\end{proof}

Consider the function $\Psi(\cdot,\cdot): \mathbb V\times\mathbb V \mapsto \R$, \begin{align*}
\Psi(\nu_1,\mu_1,\nu_2,\mu_2) := 
\int_{\X} \ell(x,\nu_1)\nu_2(dx)+ \int_{\X\times \G}g(x,\nu_1)\mu_2(dx,dy).
\end{align*}
Notice that the function does not depend on the variable $\mu_1$; nevertheless, we include it here for technical reasons.

\begin{lemma}
The mapping $\Psi$ is continuous.
\end{lemma}
\begin{proof}
It is sufficient to show sequential continuity because $\mathbb V$ is metrizable. Take a sequence $(\nu^1_n,\mu^1_n,\nu^2_n,\mu^2_n)\to(\nu^1,\mu^1,\nu^2,\mu^2) \in\mathbb V\times \mathbb V$.
Then,

\begin{align*}
    &|\Psi(\nu^1_n,\mu^1_n,\nu^2_n,\mu^2_n)-\Psi(\nu^1,\mu^1,\nu^2,\mu^2)|\\&\quad\leq\Big|\int_{\X} \ell(x,\nu^1_n)\nu^2_n(dx) - \int_{\X} \ell(x,\nu^1)\nu^2_n(dx)\Big|+\Big|\int_{\X} \ell(x,\nu^1)\nu^2_n(dx)-\int_{\X} \ell(x,\nu^1)\nu^2(dx)\Big|\\&\qquad+\Big|\int_{\X\times \G}g(x,\nu_n^1)\mu_n^2(dx,dy)-\int_{\X\times \G}g(x,\nu^1)\mu_n^2(dx,dy)\Big|\\&\qquad+\Big|\int_{\X\times \G}g(x,\nu^1)\mu_n^2(dx,dy)-\int_{\X\times \G}g(x,\nu^1)\mu^2(dx,dy)\Big|\\
    &\quad\leq L_\ell\cdot d_P(\nu^1_n,\nu^1)+\Big|\int_{\X} \ell(x,\nu^1)\nu^2_n(dx) - \int_{\X} \ell(x,\nu^1)\nu^2(dx)\Big|\\
    &\qquad+d_3\max_{i\in\lbrs{d_3}}\Big\{L_{g_i}\cdot d_P(\nu_n^1,\nu^1)\Big\}\mu^2_{n,i}(\X\times\G)\\&\qquad+d_3\max_{i\in\lbrs{d_3}}\Big|\int_{\X\times \G}g_i(x,\nu^1)\mu_{n,i}^2(dx,dy)-\int_{\X\times \G}g_i(x,\nu^1)\mu_{i}^2(dx,dy)\Big|\\
    &\quad\to 0.
\end{align*}
The first inequality is the triangle inequality. The second inequality follows by Assumption $(C_6)$, which assumes the Lipschitz continuity of $\ell$ and $g$, where $L_\ell$ and $L_{g_i}$ are the Lipschitz constants. For the limit, the first term $\to 0$ by the convergence of the sequence in $\mathbb V$; the second term $\to 0$ by weak convergence, since by $(C_5)$, $\ell$ is continuous and bounded on $\X$; the third term $\to 0$ by the convergence of the sequence in $\mathbb V$, and the fact that
\begin{align*}
    \sup_n\mu^2_{n,i}(\X\times\G) < \iy,
\end{align*}
which is indicated by \eqref{Y:finite}; the fourth term $\to 0$ by weak convergence as well, because $g$ is also bounded and continuous.
\end{proof}

Next, we present a special case of the well-known Berges’s maximum theorem, see e.g., \cite[17.31]{MR2378491}. 
\begin{lemma}\label{Berge}
Let $E$ be a metric space, and $K$ a compact metric space, and $\Psi:E\times K \mapsto \mathbb R$ a continuous function. Then $\gam(e):=\max_{k\in K}\Psi(e,k)$ is continuous, and the following set-valued function is upper hemicontinuous and compact-valued:
\begin{align*}
    E\ni e \mapsto \argmax_{k\in K}\Psi(e,k):=\{k \in L: \Psi(e,k) = \gam(e)\} \in 2^K.
\end{align*}
\end{lemma}
To apply this lemma to our case, define the set-valued function $\Psi^*:\mathbb V \mapsto 2^{\mathbb V}$ by
\begin{align*}
    \Psi^*(\nu_1,\mu_1) = \argmax_{(\nu_2,\mu_2) \in \mathbb V} \Psi(\nu_1,\mu_1,\nu_2,\mu_2).
\end{align*}
The $\argmax$ is well-defined since by fixing $e = (\nu_1, \mu_1)$, we may use Assumption \ref{ass:MFG} 
$(C_5)$ to recover Assumption \ref{ass:22:max}. Hence, Theorem \ref{thm:main:max} guarantees that the optimum exists.

From the above lemma, we get that $\Psi^*$ is upper hemicontinuous and compact-valued.
The function $\Psi$ is linear in $(\nu_2,\mu_2)$ and therefore the function $\Psi^*$ is convex-valued in the sense that for each $(\nu_1,\mu_1)\in\mathbb V$, $\Psi^*(\nu_1,\mu_1)$ is a convex set. 
\begin{definition}
Let $E$ and $K$ be two topological vector spaces, and $\Psi^*: E\mapsto 2^K$ be a set-valued function. If $K$ is convex, $\Psi^*$ is upper hemicontinuous, and for all $e\in E$, $\Psi^*(e)$ is non-empty, compact, and convex, then $\Psi^*$ is called a {\rm Kakutani map}.
\end{definition}
We see that our function $\Psi^*$ is a Kakutani map. We will then use the following fixed point theorem provided in \cite[II,7.8.6]{MR1987179}.
\begin{lemma}(Kakutani--Glicksberg--Fan theorem)
Let $C$ be a compact convex subset of a locally convex space $E$, and let $S : C \mapsto 2^C$ be a Kakutani map. Then, $S$ has a fixed point; namely, there is $c\in C$, such that $c\in S(c)$. 
\end{lemma}

\begin{proof}[Proof of Theorem \ref{exi:LP}]
The set $\mathbb V$ is a compact and convex subset of $\calP({\X})\times\calM({{\X}\times\G})$, which is clearly convex, and the map $\Psi^*$ is a Kakutani map. Therefore, there exists a fixed point of the map $\Psi^*$; call it $({\hat \nu},{\hat\mu})$. Furthermore, for any $\nu_1\in\calP(\X)$ and $\mu_1\in\calI(\nu_1)$,

\begin{align*}
    \max_{(\nu_2,\mu_2) \in \mathbb V} \Psi(\nu_1,\mu_1,\nu_2,\mu_2) =\sup_{\nu\in\mathbb A,\;\mu \in \calI(\nu)}\Big\{\int_{\X} \ell(x,\nu_1)\nu(dx)+ \int_{\X\times \G}g(x,\nu_1)\mu(dx,dy)\Big\} = \calK(\nu_1).
\end{align*}
We have,

\begin{align*}
    \int_{\X} \ell(x,\hat\nu)\hat\nu(dx)+\int_{\X\times \G}g(x,\hat\nu)\hat\mu(dx,dy) = \Psi({\hat \nu},{\hat\mu},{\hat \nu},{\hat\mu}) = \max_{(\nu_2,\mu_2) \in \mathbb V}\Psi({\hat \nu},{\hat\mu},\nu_2,\mu_2) = \calK(\hat\nu).
\end{align*}
So $\hat\nu$ is an MFG equilibrium.

\end{proof}

\vspace{5pt}
{\bf Acknowledgment.} We thank the anonymous AE and the referees for their insightful comments, which helped us improve our paper.

\footnotesize

\bibliographystyle{abbrv} 
\bibliography{refs} 

\begin{thebibliography}{10}

\bibitem{Adlakha}
S.~Adlakha, R.~Johari, and G.~Weintraub.
\newblock Equilibria of dynamic games with many players: Existence,
  approximation, and market structure.
\newblock {\em Journal of Economic Theory}, 156, 11 2015.

\bibitem{aid2023stationary}
R.~Aid, M.~Basei, and G.~Ferrari.
\newblock A stationary mean-field equilibrium model of irreversible investment
  in a two-regime economy, 2023.

\bibitem{MR2378491}
C.~D. Aliprantis and K.~C. Border.
\newblock {\em Infinite dimensional analysis}.
\newblock Springer, Berlin, third edition, 2006.
\newblock A hitchhiker's guide.

\bibitem{e2018class}
L.~H. Alvarez~E.
\newblock A class of solvable stationary singular stochastic control problems,
  2018.

\bibitem{alv-hen2019}
L.~H. Alvarez~E. and A.~Hening.
\newblock Optimal sustainable harvesting of populations in random environments.
\newblock {\em Stochastic Process. Appl.}, 2019.

\bibitem{Anahtarci2019ValueIA}
B.~Anahtarci, C.~D. Kariksiz, and N.~Saldi.
\newblock Value iteration algorithm for mean-field games.
\newblock {\em Syst. Control. Lett.}, 143:104744, 2019.

\bibitem{Ari17}
A.~Arapostathis, A.~Biswas, and J.~Carroll.
\newblock On solutions of mean field games with ergodic cost.
\newblock {\em J. Math. Pures Appl. (9)}, 107(2):205--251, 2017.

\bibitem{Pang15}
A.~Arapostathis, A.~Biswas, and G.~Pang.
\newblock Ergodic control of multi-class {$M/M/N+M$} queues in the
  {H}alfin-{W}hitt regime.
\newblock {\em Ann. Appl. Probab.}, 25(6):3511--3570, 2015.

\bibitem{Ari_book}
A.~Arapostathis, V.~S. Borkar, and M.~K. Ghosh.
\newblock {\em Ergodic Control of Diffusion Processes}.
\newblock Encyclopedia of Mathematics and its Applications. Cambridge
  University Press, 2011.

\bibitem{Pang19}
A.~Arapostathis, H.~Hmedi, G.~Pang, and N.~Sandri\'{c}.
\newblock Uniform polynomial rates of convergence for a class of
  {L}\'{e}vy-driven controlled {SDE}s arising in multiclass many-server queues.
\newblock In {\em Modeling, stochastic control, optimization, and
  applications}, volume 164 of {\em IMA Vol. Math. Appl.}, pages 1--20.
  Springer, Cham, 2019.

\bibitem{Pang16}
A.~Arapostathis and G.~Pang.
\newblock Ergodic diffusion control of multiclass multi-pool networks in the
  {H}alfin-{W}hitt regime.
\newblock {\em Ann. Appl. Probab.}, 26(5):3110--3153, 2016.

\bibitem{Pang18}
A.~Arapostathis and G.~Pang.
\newblock Infinite-horizon average optimality of the {N}-network in the
  {H}alfin-{W}hitt regime.
\newblock {\em Math. Oper. Res.}, 43(3):838--866, 2018.

\bibitem{ata2023singular}
B.~Ata, J.~M. Harrison, and N.~Si.
\newblock Singular control of (reflected) brownian motion: A computational
  method suitable for queueing applications, 2023.

\bibitem{Atar-Budh-Will-07}
R.~Atar, A.~Budhiraja, and R.~J. Williams.
\newblock H{JB} equations for certain singularly controlled diffusions.
\newblock {\em Ann. Appl. Probab.}, 17(5-6):1745--1776, 2007.

\bibitem{aydin2023robustness}
U.~Aydın and N.~Saldi.
\newblock Robustness and approximation of discrete-time mean-field games under
  discounted cost criterion, 2023.

\bibitem{Bardi14}
M.~Bardi and F.~S. Priuli.
\newblock Linear-quadratic {$N$}-person and mean-field games with ergodic cost.
\newblock {\em SIAM J. Control Optim.}, 52(5):3022--3052, 2014.

\bibitem{biswas2017}
A.~Biswas.
\newblock An ergodic control problem for many-server multiclass queueing
  systems with cross-trained servers.
\newblock {\em Stoch. Syst.}, 7(2):263--288, 2017.

\bibitem{Borkar}
V.~S. Borkar.
\newblock {\em Optimal control of diffusion processes}, volume 203 of {\em
  Pitman Research Notes in Mathematics Series}.
\newblock Longman Scientific \& Technical, Harlow; copublished in the United
  States with John Wiley \& Sons, Inc., New York, 1989.

\bibitem{handBM}
A.~N. Borodin and P.~Salminen.
\newblock {\em Handbook of Brownian Motion - Facts and Formulae}.
\newblock Probability and Its Applications. Springer, 2002.

\bibitem{Roxana2020}
G.~Bouveret, R.~Dumitrescu, and P.~Tankov.
\newblock Mean-field games of optimal stopping: a relaxed solution approach.
\newblock {\em SIAM J. Control Optim.}, 58(4):1795--1821, 2020.

\bibitem{Budh-03}
A.~Budhiraja.
\newblock An ergodic control problem for constrained diffusion processes:
  existence of optimal {M}arkov control.
\newblock {\em SIAM J. Control Optim.}, 42(2):532--558, 2003.

\bibitem{bud-gho-chi2011}
A.~Budhiraja, A.~P. Ghosh, and C.~Lee.
\newblock Ergodic rate control problem for single class queueing networks.
\newblock {\em SIAM J. Control Optim.}, 49(4):1570--1606, 2011.

\bibitem{bud-ros2006}
A.~Budhiraja and K.~Ross.
\newblock Existence of optimal controls for singular control problems with
  state constraints.
\newblock {\em Ann. Appl. Probab.}, 16(4):2235--2255, 2006.

\bibitem{finite-fuel}
L.~Campi, T.~De~Angelis, M.~Ghio, and G.~Livieri.
\newblock Mean-field games of finite-fuel capacity expansion with singular
  controls.
\newblock {\em Ann. Appl. Probab.}, 32(5):3674--3717, 2022.

\bibitem{cannerozzi2024cooperation}
F.~Cannerozzi and G.~Ferrari.
\newblock Cooperation, correlation and competition in ergodic $n$-player games
  and mean-field games of singular controls: A case study, 2024.

\bibitem{HYCao}
H.~Cao, J.~Dianetti, and G.~Ferrari.
\newblock Stationary discounted and ergodic mean field games with singular
  controls.
\newblock {\em Math. Oper. Res.}, 48(4):1871--1898, 2023.

\bibitem{MFGrever}
H.~Cao and X.~Guo.
\newblock M{FG}s for partially reversible investment.
\newblock {\em Stochastic Process. Appl.}, 150:995--1014, 2022.

\bibitem{Cao-Guo-Lee}
H.~Cao, X.~Guo, and J.~S. Lee.
\newblock Approximation of {$ N $}-player stochastic games with singular
  controls by mean field games.
\newblock {\em Numer. Algebra Control Optim.}, 13(3-4):604--629, 2023.

\bibitem{car-por}
P.~Cardaliaguet and A.~Porretta.
\newblock Long time behavior of the master equation in mean-field game theory.
\newblock {\em Analysis \& PDE}, 2017.

\bibitem{car2020}
R.~Carmona.
\newblock Applications of mean field games in financial engineering and
  economic theory.
\newblock In {\em Mean field games}, volume~78 of {\em Proc. Sympos. Appl.
  Math.}, pages 165--219. Amer. Math. Soc., Providence, RI, [2021] \copyright
  2021.

\bibitem{CD1}
R.~Carmona and F.~Delarue.
\newblock {\em Probabilistic theory of mean field games with applications.
  {I}}, volume~83 of {\em Probability Theory and Stochastic Modelling}.
\newblock Springer, Cham, 2018.
\newblock Mean field FBSDEs, control, and games.

\bibitem{CD2}
R.~Carmona and F.~Delarue.
\newblock {\em Probabilistic theory of mean field games with applications.
  {II}}, volume~84 of {\em Probability Theory and Stochastic Modelling}.
\newblock Springer, Cham, 2018.
\newblock Mean field games with common noise and master equations.

\bibitem{car-lac2017}
R.~Carmona, F.~Delarue, and D.~Lacker.
\newblock Mean field games of timing and models for bank runs.
\newblock {\em Appl. Math. Optim.}, 76(1):217--260, 2017.

\bibitem{coh2019a}
A.~Cohen.
\newblock Brownian control problems for a multiclass {M}/{M}/1 queueing problem
  with model uncertainty.
\newblock {\em Math. Oper. Res.}, 44(2):739--766, 2019.

\bibitem{coh2019b}
A.~Cohen.
\newblock On singular control problems, the time-stretching method, and the
  weak-{M}1 topology.
\newblock {\em SIAM J. Control Optim.}, 59(1):50--77, 2021.

\bibitem{cohen2021optimal}
A.~Cohen, A.~Hening, and C.~Sun.
\newblock Optimal ergodic harvesting under ambiguity.
\newblock {\em SIAM J. Control Optim.}, 60(2):1039--1063, 2022.

\bibitem{CZ2022}
A.~Cohen and E.~Zell.
\newblock Analysis of the finite-state ergodic master equation.
\newblock {\em Appl. Math. Optim.}, 87(3):Paper No. 40, 53, 2023.

\bibitem{davis1990portfolio}
M.~H. Davis and A.~R. Norman.
\newblock Portfolio selection with transaction costs.
\newblock {\em Math. Oper. Res.}, 15(4):676--713, 1990.

\bibitem{unifying}
J.~Dianetti, G.~Ferrari, M.~Fischer, and M.~Nendel.
\newblock A unifying framework for submodular mean field games.
\newblock {\em Math. Oper. Res.}, 48(3):1679--1710, 2023.

\bibitem{dianetti2023ergodic}
J.~Dianetti, G.~Ferrari, and I.~Tzouanas.
\newblock Ergodic mean-field games of singular control with regime-switching
  (extended version), 2023.

\bibitem{Roxana2021}
R.~Dumitrescu, M.~Leutscher, and P.~Tankov.
\newblock Control and optimal stopping mean field games: a linear programming
  approach.
\newblock {\em Electron. J. Probab.}, 26:Paper No. 157, 49, 2021.

\bibitem{echeveria}
P.~Echeverr\'{\i}a.
\newblock A criterion for invariant measures of {M}arkov processes.
\newblock {\em Z. Wahrsch. Verw. Gebiete}, 61(1):1--16, 1982.

\bibitem{eli-hub-tur2020}
R.~Elie, E.~Hubert, and G.~Turinici.
\newblock Contact rate epidemic control of {COVID}-19: an equilibrium view.
\newblock {\em Math. Model. Nat. Phenom.}, 15:Paper No. 35, 25, 2020.

\bibitem{Feleqi13}
E.~Feleqi.
\newblock The derivation of ergodic mean field game equations for several
  populations of players.
\newblock {\em Dyn. Games Appl.}, 3(4):523--536, 2013.

\bibitem{Fu2023}
G.~Fu.
\newblock Extended mean field games with singular controls.
\newblock {\em SIAM J. Control Optim.}, 61(1):283--312, 2023.

\bibitem{Fu21}
G.~Fu and U.~Horst.
\newblock Mean field games with singular controls.
\newblock {\em SIAM J. Control Optim.}, 55(6):3833--3868, 2017.

\bibitem{diogo15}
D.~A. Gomes and H.~Mitake.
\newblock Existence for stationary mean-field games with congestion and
  quadratic {H}amiltonians.
\newblock {\em NoDEA Nonlinear Differential Equations Appl.}, 22(6):1897--1910,
  2015.

\bibitem{gra2016}
P.~J. Graber.
\newblock Linear quadratic mean field type control and mean field games with
  common noise, with application to production of an exhaustible resource.
\newblock {\em Appl. Math. Optim.}, 74(3):459--486, 2016.

\bibitem{MR1987179}
A.~Granas and J.~Dugundji.
\newblock {\em Fixed point theory}.
\newblock Springer Monographs in Mathematics. Springer-Verlag, New York, 2003.

\bibitem{NMFG19}
X.~Guo and R.~Xu.
\newblock Stochastic games for fuel follower problem: {$N$} versus mean field
  game.
\newblock {\em SIAM J. Control Optim.}, 57(1):659--692, 2019.

\bibitem{Instantaneous}
J.~M. Harrison and M.~I. Taksar.
\newblock Instantaneous control of {B}rownian motion.
\newblock {\em Math. Oper. Res.}, 8(3):439--453, 1983.

\bibitem{RBM_Harrison}
J.~M. Harrison and R.~J. Williams.
\newblock Brownian models of open queueing networks with homogeneous customer
  populations.
\newblock {\em Stochastics}, 22(2):77--115, 1987.

\bibitem{MRBM-sta}
J.~M. Harrison and R.~J. Williams.
\newblock Multidimensional reflected {B}rownian motions having exponential
  stationary distributions.
\newblock {\em Ann. Probab.}, 15(1):115--137, 1987.

\bibitem{Hau71}
U.~G. Haussmann.
\newblock Optimal stationary control with state control dependent noise.
\newblock {\em SIAM Journal on Control}, 9(2):184--198, 1971.

\bibitem{Haussmann1}
U.~G. Haussmann and W.~Suo.
\newblock Existence of singular optimal control laws for stochastic
  differential equations.
\newblock {\em Stochastics Stochastics Rep.}, 48(3-4):249--272, 1994.

\bibitem{Haussmann2}
U.~G. Haussmann and W.~Suo.
\newblock Singular optimal stochastic controls. {I}. {E}xistence.
\newblock {\em SIAM J. Control Optim.}, 33(3):916--936, 1995.

\bibitem{Haussmann3}
U.~G. Haussmann and W.~Suo.
\newblock Singular optimal stochastic controls. {II}. {D}ynamic programming.
\newblock {\em SIAM J. Control Optim.}, 33(3):937--959, 1995.

\bibitem{coexist}
A.~Hening and D.~H. Nguyen.
\newblock Coexistence and extinction for stochastic {K}olmogorov systems.
\newblock {\em Ann. Appl. Probab.}, 28(3):1893--1942, 2018.

\bibitem{HNUK19}
A.~Hening, D.~H. Nguyen, S.~C. Ungureanu, and T.~K. Wong.
\newblock Asymptotic harvesting of populations in random environments.
\newblock {\em J. Math. Biol.}, 78(1-2):293--329, 2019.

\bibitem{hen-qua2020}
A.~Hening and K.~Q. Tran.
\newblock Harvesting and seeding of stochastic populations: analysis and
  numerical approximation.
\newblock {\em J. Math. Biol.}, 81(1):65--112, 2020.

\bibitem{Huang2007}
M.~Huang, P.~E. Caines, and R.~P. Malham{\'e}.
\newblock The {N}ash certainty equivalence principle and {M}ckean-{V}lasov
  systems: {A}n invariance principle and entry adaptation.
\newblock In {\em Decision and Control, 2007 46th IEEE Conference on}, pages
  121--126. IEEE, 2007.

\bibitem{Huang2006}
M.~Huang, R.~P. Malham{\'e}, and P.~E. Caines.
\newblock Large population stochastic dynamic games: Closed-loop
  {M}c{K}ean-{V}lasov systems and the {N}ash certainty equivalence principle.
\newblock {\em Commun. Inf. Syst.}, 6(3):221--251, 2006.

\bibitem{weining}
W.~Kang and K.~Ramanan.
\newblock Characterization of stationary distributions of reflected diffusions.
\newblock {\em Ann. Appl. Probab.}, 24(4):1329--1374, 2014.

\bibitem{KSbook}
I.~Karatzas and S.~E. Shreve.
\newblock {\em Brownian motion and stochastic calculus}, volume 113 of {\em
  Graduate Texts in Mathematics}.
\newblock Springer-Verlag, New York, second edition, 1991.

\bibitem{kha2}
R.~Khasminskii.
\newblock {\em Stochastic stability of differential equations}, volume~66 of
  {\em Stochastic Modelling and Applied Probability}.
\newblock Springer, Heidelberg, second edition, 2012.
\newblock With contributions by G. N. Milstein and M. B. Nevelson.

\bibitem{klein69}
D.~Kleinman.
\newblock Optimal stationary control of linear systems with control-dependent
  noise.
\newblock {\em IEEE Transactions on Automatic Control}, 14(6):673--677, 1969.

\bibitem{kruk1}
L.~Kruk.
\newblock Optimal policies for n-dimensional singular stochastic control
  problems part i: The skorokhod problem.
\newblock {\em SIAM Journal on Control and Optimization}, 38(5):1603--1622,
  2000.

\bibitem{kruk2}
L.~Kruk.
\newblock Optimal policies for n-dimensional singular stochastic control
  problems. part ii: The radially symmetric case. ergodic control.
\newblock {\em SIAM Journal on Control and Optimization}, 39(2):635--659, 2000.

\bibitem{kurtz98}
T.~G. Kurtz and R.~H. Stockbridge.
\newblock Existence of {M}arkov controls and characterization of optimal
  {M}arkov controls.
\newblock {\em SIAM J. Control Optim.}, 36(2):609--653, 1998.

\bibitem{kur-sto}
T.~G. Kurtz and R.~H. Stockbridge.
\newblock Stationary solutions and forward equations for controlled and
  singular martingale problems.
\newblock {\em Electron. J. Probab.}, 6:no. 17, 52, 2001.

\bibitem{kurtz2017linear}
T.~G. Kurtz and R.~H. Stockbridge.
\newblock Linear programming formulations of singular stochastic control
  problems: Time-homogeneous problems, 2017.

\bibitem{MFG1}
J.-M. Lasry and P.-L. Lions.
\newblock Mean field games.
\newblock {\em Jpn. J. Math.}, 2(1):229--260, 2007.

\bibitem{Vijay}
J.~Li, B.~Xia, X.~Geng, H.~Ming, S.~Shakkottai, V.~Subramanian, and L.~Xie.
\newblock Mean field games in nudge systems for societal networks.
\newblock {\em ACM Trans. Model. Perform. Eval. Comput. Syst.}, 3(4), Aug.
  2018.

\bibitem{li-rep-sir2019}
Z.~Li, A.~M. Reppen, and R.~Sircar.
\newblock A mean field games model for cryptocurrency mining.
\newblock {\em Management Science}, 0(0):null, 0.

\bibitem{lia-zer2020}
G.~{Liang} and M.~{Zervos}.
\newblock {Ergodic singular stochastic control motivated by the optimal
  sustainable exploitation of an ecosystem}.
\newblock {\em arXiv e-prints}, page arXiv:2008.05576, Aug. 2020.

\bibitem{MFG2}
P.-L. Lions and J.-M. Lasry.
\newblock Large investor trading impacts on volatility.
\newblock In {\em Paris-{P}rinceton {L}ectures on {M}athematical {F}inance
  2004}, volume 1919 of {\em Lecture Notes in Math.}, pages 173--190. Springer,
  Berlin, 2007.

\bibitem{menaldi-robin}
J.~L. Menaldi and M.~Robin.
\newblock Singular ergodic control for multidimensional {G}aussian-{P}oisson
  processes.
\newblock {\em Stochastics}, 85(4):682--691, 2013.

\bibitem{Neumann20}
B.~A. Neumann.
\newblock Stationary equilibria of mean field games with finite state and
  action space.
\newblock {\em Dyn. Games Appl.}, 10(4):845--871, 2020.

\bibitem{RL22}
B.~Pang and Z.-P. Jiang.
\newblock Reinforcement learning for adaptive optimal stationary control of
  linear stochastic systems.
\newblock {\em IEEE Transactions on Automatic Control}, 68(4):2383--2390, 2023.

\bibitem{reppen2018singular}
M.~Reppen.
\newblock {\em Singular Control in Financial Economics}.
\newblock PhD thesis, ETH Zurich, 2018.

\bibitem{MR94838}
V.~S. Varadarajan.
\newblock Weak convergence of measures on separable metric spaces.
\newblock {\em Sankhy\={a}}, 19:15--22, 1958.

\bibitem{ref-OU}
A.~Ward and P.~Glynn.
\newblock Properties of the reflected ornstein-uhlenbeck process.
\newblock {\em Queueing Syst.}, 44:109--123, 06 2003.

\bibitem{Lei}
L.~Ying.
\newblock Stein's method for mean field approximations in light and heavy
  traffic regimes.
\newblock {\em Proc. ACM Meas. Anal. Comput. Syst.}, 1(1), June 2017.

\end{thebibliography}

\end{document}